\documentclass[11pt]{article}
\usepackage{amsfonts,amssymb,amsmath}
\usepackage[english]{babel}
\usepackage{a4wide}
\usepackage{amssymb}
\usepackage{amsthm}
\usepackage{xcolor}
\usepackage[cal=boondoxo]{mathalfa}
\usepackage{pdfsync}
\usepackage{color}
\usepackage{bm}
\usepackage{dsfont}
\usepackage{marginnote}
\usepackage{numbysec}
\usepackage{comment}

\usepackage[backref]{hyperref}

\usepackage{cleveref}

\usepackage[misc,geometry]{ifsym}
\usepackage{graphicx}

\newcommand{\si}{{\sigma}}
\newcommand{\al}{{\alpha}}
\newcommand{\eps}{{\varepsilon}}
\newcommand{\om}{{\omega}}
\newcommand{\bbQ}{{\mathbb Q}}
\newcommand{\bbZ}{{\mathbb Z}}
\newcommand{\bbR}{{\mathbb R}}
\newcommand{\bbP}{{\mathbb P}}
\newcommand{\bbE}{{\mathbb E}}
\newcommand{\bbT}{{\mathbb T}}
\newcommand{\bbN}{{\mathbb N}}

\newcommand{\la}{\lambda}

\newcommand{\Om}{{\Omega}}

\newcommand{\tk}[1]{{\color{blue}#1}}

\newtheorem{thm}{Theorem}[section]
\newtheorem{lemma}[thm]{Lemma}
\newtheorem{proposition}[thm]{Proposition}

\newtheorem{hypothesis}[thm]{Hypothesis}

\title{Homogenization of   stable-like
    operators with random, ergodic coefficients}

\date{\today}

\author {Tomasz Klimsiak \textsuperscript{1,2} \and  Tomasz Komorowski
  \textsuperscript{1,{\tiny\Letter}} \and Lorenzo Marino \textsuperscript{3}}

\begin{document}
\numberbysection

\nocite{*}

\maketitle
\begin{center}
 \par \bigskip
\textsuperscript{\tiny 1} {\scriptsize Institute of Mathematics, Polish Academy of Sciences,\\
 \'{S}niadeckich 8,   00-656 Warsaw, Poland, email: {\tt
   tkomorowski@impan.pl}, corresponding author} \par \medskip

\textsuperscript{\tiny 2} {\scriptsize Faculty of
Mathematics and Computer Science, Nicolaus Copernicus University,\\ 
Chopina 12/18, 87-100 Toru\'n, Poland, email:{\tt  tomas@mat.umk.pl}} \par \medskip

\textsuperscript{\tiny 3} {\scriptsize Unit\'e de Math\'ematiques Appliqu\'ees (UMA), ENSTA Paris,\\
828 Boulevard des Mar\'echaux, F-91120 Palaiseau, France, email: {\tt lorenzo.marino@ensta-paris.fr}} \par
\end{center}

\vspace*{-.1in}

\begin{abstract}
We  show  homogenization for
a family of
$\bbR^d$-valued
stable-like processes $(X_t^{\eps;\theta})_{t\ge0}$,
$\eps\in(0,1]$,
whose
(random) Fourier symbols equal
$q_\eps(x,\xi;\theta)=\frac{1}{\eps^{\al}}q\Big(\frac{x}{\eps},\eps \xi; \theta\Big)$, where
$$  q(x,\xi; \theta)=\int_{\mathbb
    R^d}\left(1-e^{i  y\cdot\xi}+i y \cdot \xi\mathds{1}_{\{|y|\le1\}}\right)\,\frac{\langle a(x; \theta)y,y\rangle}{|y|^{d+2+\alpha}}\,dy,
  $$
for $(x,\xi, \theta)\in \bbR^{2d}\times\Theta$.  Here $\al\in(0,2)$ and the family  
$(a(x; \theta))_{x\in \bbR^d}$ of $d\times d$ symmetric, non-negative definite matrices  is a
stationary ergodic random field   over some probability space $(\Theta,{\cal
  H},m)$. 
  We assume that the random field is deterministically bounded and non-degenerate,
  i.e.\ $|a(x; \theta)|\le \Lambda$ and   $\text{Tr}(a(x; \theta))\ge \lambda$
 for some $\Lambda,\lambda>0$ and all $\theta\in \Theta$. In addition, we
 suppose that the field is
    regular enough so  that
for any $\theta\in\Theta$, the operator $-q(\cdot,D; \theta)$, defined on the
space of compactly supported $C^2$ functions on $\bbR^d$,
is closable in   the space of continuous
functions vanishing at infinity    and its closure generates a Feller
semigroup.
We prove the weak convergence of the laws of
$(X_t^{\eps; \theta})_{t\ge0}$, as $\eps\downarrow0$, in the
Skorokhod space,  $m$-a.s.\ in $\theta$, to an $\al$-stable process whose
Fourier symbol $\bar q(\xi)$ is given by 
$
\bar q(\xi)=\int_{\Om} q(0,\xi; \theta)\Phi_*(\theta)\,m(d \theta),$ 
where 
$\Phi_*$ is a strictly positive density w.r.t. measure $m$.  Our result has an   analytic
interpretation in terms of the convergence, as $\eps\downarrow0$, of the
solutions  to random integro-differential equations
$  \partial_tu_\eps(t,x; \theta)=-q_\eps(x,D; \theta) u_\eps(t,x; \theta)$, with
the initial condition $u_\eps(0,x; \theta)=f(x)$, where $f$ is a bounded
and continuous function on $\bbR^d$.

\end{abstract}


\section{Introduction}
\label{sec1}


In the present paper, we consider the problem of homogenization for a
class of  stable-like  operators  with random coefficients.
More precisely, let $(\Theta,{\cal H},m)$ be a   probability
space. For each $\theta\in\Theta$, we assign an integral operator
brown $L^\theta$ on $C_c^2(\bbR^d)$ - the set of all $C^2$-smooth,
compactly supported functions on $\bbR^d$ - given by
\begin{equation}\label{Lom}
L^\theta u(x) \, := \frac12\int_{\bbR^d}\left[u(x+z) +u(x-z)-2u(x)\right]\,{\rm n}(x,dz;\theta),
\end{equation}
where the L\'evy kernel ${\rm n}$ has  the form 
\begin{equation}
  \label{Lt}
  {\rm n}(x,d{\color{brown}z};\theta):= {\frac{\langle  a(x;\theta){
        z},{ z}\rangle}{| z|^{d+\alpha+2}}}\,d{ z},
  \end{equation}
for some $\alpha \in(0,2)$ and  $a(x;\theta)=[a_{i,j}(x;\theta)]_{i,j=1}^d$ is a  stationary and
ergodic random field taking values in $\mathcal S^+_d$, the family of all non-negative definite, symmetric  $d\times d$-matrices.

Throughout the paper, we assume the following:
\begin{hypothesis}
  \label{H1}
The realizations of the random field {$\left(a(x;\theta)\right)_{x\in \bbR^d}$} are continuous in $x$, deterministically bounded, both in $x$ and $ \theta$, and satisfy the  non-degeneracy  condition, i.e\ there exist $\lambda,\Lambda>0$ such that
\begin{equation}
\label{eq.MH1}
\|a(x;\theta)\|\le \Lambda,\qquad 
\operatorname{Tr}(a(x;\theta)) \ge \lambda, \qquad (x,\theta)\in \bbR^d\times \Theta.  
\end{equation} 
Here, $ \|a \|:= \sum_{i,j=1}^d|a_{i,j}|$.
\end{hypothesis}

\begin{hypothesis}
  \label{H2} For any   $\theta\in\Theta$, the martingale  
problem (see Section \ref{sec2.2} below) associated with  the operator
{$L^\theta$ defined in} \eqref{Lom}
is well-posed on the Skorokhod space ${\mathcal
  D}$ of  {all} $\bbR^d$-valued   c\`adl\`ag paths equipped with the
$J_1$-topology (see \cite[Section 12]{bil} for the definition of the topology).
\end{hypothesis}
Hypothesis \ref{H2} holds e.g. whenever $a(x;\theta)$ is sufficiently smooth in $x$ (see \cite[Theorem 4.2]{Hoh}) 
or it is Lipschitz continuous in $x$ and uniformly elliptic (see
\cite[Theorem 1.3]{BKS}), i.e. or some ${\cal r}>0$
\begin{equation}
\label{eq.uel} 
{ \langle a(x;\theta)\xi,\xi\rangle\ge{\cal{r}} |\xi|^2}, \qquad (x,\xi,\theta)\in\bbR^{2d}\times\Theta.
\end{equation}

{Assume Hypothesis \ref{H1}. The results \cite[Theorem 1.1, Lemma
  2.1]{kuhn}} and \cite[Theorem 4.10.3]{Kol} show that then Hypothesis \ref{H2} is
equivalent to the following condition:

\begin{itemize}
  \item[-]
for any $\theta\in \Theta$, the operator $L^\theta$ defined in
\eqref{Lom} is closable in $C_0(\bbR^d)$ - the set of all continuous
functions on $\bbR^d$ vanishing at infinity - and its closure
generates a Feller semigroup.
\end{itemize}
{In particular, it follows from the above that} for each $\theta\in\Theta$ there exists a
unique $\bbR^d$-valued, c\`adl\`ag,  Feller  process 
$(X^{x;\theta}_t; t\ge 0,x\in\bbR^d)$, defined over some probability space $(\Sigma,\mathcal{A},\bbP)$,
satisfying $X^{x;\theta}_0=x$, $\bbP$-a.s. Its random Fourier
symbol  {$q\colon \bbR^{d}\times\bbR^{d}\times\Theta\to \bbR$} equals (due to the symmetry of the measure ${\rm n}(x,\cdot,\theta)$)
\begin{equation}
  \label{beq01}
  q(x,\xi;\theta)=\int_{\mathbb
    R^d}(1 -e^{iz\cdot\xi}+i z \cdot \xi\mathds{1}_{\{|z|\le1\}})\,{\rm
    n}(x,dz;\theta)=\int_{\mathbb
    R^d}(1 -\cos(z\cdot\xi))\,{\rm n}(x,dz;\theta).
\end{equation}

 {To study the homogenization problem, we introduce the scaled processes
$X_t^{\eps,x;\theta}(\zeta):=\eps X^{x/\eps;\theta}_{t\eps^{-\al}}(\zeta)$,
$t\ge0$, defined  for any  $\eps>0$. The processes  are
 considered  over the product probability space $(\Theta\times\Sigma,{\cal
    H}\otimes{\cal A},m\otimes\bbP)$.
For each $\theta\in\Theta$, the process $(X_t^{\eps,x;\theta},t\ge 0,x\in\bbR^d)$
is Feller and its infinitesimal generator on $C^2_c(\bbR^d)$ equals
\[
L^\theta_\varepsilon u(x):= \frac12\int_{\bbR^d}\left[u(x+z)-2u(x)+u(x-z)\right] \frac{\langle  a(x/\varepsilon;\theta)z,z\rangle}{|z|^{d+\alpha+2}}\,dz,\qquad x\in \bbR^d.\]

Our main result is the following quenched version of the convergence
of the laws of the scaled processes.

\begin{thm}
\label{th.intr1}
Under the assumptions spelled out in the foregoing the laws of
$(X_t^{\eps,x;\theta})_{t\ge 0}$  converge  weakly  in ${\cal D}$, as
$\eps\downarrow0$,  for $m$-a.s.\ in $\theta$, to the law of an  {$\alpha$-stable} process whose Fourier symbol equals 
\[\bar q(\xi) =\int_{\Theta}q(0,\xi;\theta)\Phi_*(\theta)\,m(d
  \theta)= { \int_{\mathbb
    R^d}\left[1 -\cos(z\cdot\xi)\right]\,{\frac{\langle  \bar az,z\rangle}{|z|^{d+\alpha+2}}}\,dz}, \qquad \xi\in\bbR^d.
\] 
Here $\Phi_*$ is a strictly positive density with respect to $m$ and
\[\bar a:= \int_{\Om} a(0;\theta)\Phi_*(\theta)\,m(d\theta).\]
\end{thm}


Our result has an obvious analytic interpretation in terms of
solutions to random integro-differential equations of the form
\[
\begin{cases}
\partial_tu_\eps(t,x;\theta)=L^\theta_\varepsilon u_\eps(t,x;\theta),\qquad t>0, \, x\in\bbR^d, \, \theta\in\Theta;\\
u_\eps(0,x;\theta)=f(x),
  \end{cases}
\]
where $f$ is a bounded and continuous function in $\bbR^d$.
As its direct consequence one can conclude that   for
$m$-a.s. $\theta$ the solutions
 $u_\eps(t,x;\theta)$ converge,  
as $\eps\downarrow0$, to the
solution $\bar u(t,x)$ of the following ``homogenized'' equation:
\[
  \begin{cases}
   \partial_t\bar u(t,x)=\bar L \bar u(t,x), \qquad t>0, \, x\in\bbR^d;\\  
   \bar u(0,x)=f(x), 
  \end{cases}
 \]
  where
\[
\bar L  u(x):= \frac12\int_{\bbR^d}\left[u(x+z)-2u(x)+u(x-z)\right]
\frac{\langle\bar a {
    z},{ z}\rangle}{|{ z}|^{d+\alpha+2}}\,d{ z}.
\]

\subsection*{Context}
{Homogenization of  solutions to  stochastic differential equations
  (S.D.E-s) and partial differential equations (P.D.E.-s) with random
  coefficients is a classical problem in both analysis and 
  the theory of stochastic processes. The
  pioneering results  in the subject  have been  almost simultaneously
  obtained   in \cite{kozlov} and   
  \cite{PV1}, where  the  problem of homogenization of solutions to the Dirichlet boundary value problem for
   elliptic equations  in a divergence form, with stationary and
ergodic coefficients has been considered. 
Since then the topic has been developed by many authors and for
various types of elliptic and parabolic differential equations  with
fast oscillating coefficients.  We refer an interested reader to the monographs 
\cite{AKM,BensPap,jikov,KLO12,Olla,Pavl,tartar} and the references therein.}

More recently, there has been a growing interest in  stochastic homogenization for classes of integro-differential equations and a related problem of  scaling limits of S.D.E-s with random coefficients, driven by general L\'{e}vy processes. Often such limits require a non-diffusive scaling and the result of the homogenization  is a L\'{e}vy process.

{We mention in this context, papers 
\cite{arizawa,barles,franke,horie,song,Kas-pyat,pyat1} for
equations with fast oscillating periodic coefficients and 
\cite{CCKW1,CCKW2,CCKW3,fujiwara,pyat2,rhodes,rhodes2} which are concerned with stochastically homogeneous random coefficients.} 
In particular, the paper \cite{fujiwara} considers the limit of the
martingale problem whose coefficients are driven by some uniquely
ergodic Markov process. The work \cite{pyat2}   deals with the
diffusive limit of non-local operators of  convolution type with
random  ergodic coefficients. The closest  case to ours is considered
in \cite{rhodes,rhodes2}, where one dimensional SDEs, driven by both
Brownian and Poisson noises, are examined. The coefficients are
stationary and ergodic fields. The principal difference between the
present paper and \cite{rhodes,rhodes2} is that the latter look at the
situation when the process describing the environment, as viewed from
vantage point of the trajectory of the solution of the SDE, has an
explicitly given invariant measure. In contrast, the main effort of
our article is to construct the invariant measure for the
aforementioned process.
The present paper   is related to
the result of \cite{PV2}   where  diffusions with no jumps
have been considered, i.e.\ $b\equiv 0$ and ${\rm n}\equiv0$. In this
case, $q_\eps(x,\xi;\theta)=\frac12a(x/\eps;\theta)\xi\cdot\xi$ for any $\varepsilon>0$.


{Finally,  we mention that in the non-linear framework,
  homogenization of solutions of integro-differential equations with an external
  Dirichlet boundary condition on a bounded domain has been considered
  in \cite{schwab,schwab2}, for a class of fully non-linear, non-local
  elliptic operators with fast ocillating, either periodic, or
  stochastic and  ergodic,  coefficients, that includes also the
  operators of the form \eqref{Lom}. However, their method of proof is
  quite different from ours. We should also emphasize that our result
 deals with the convergence of the laws of stochastic processes, that
 constitutes the novelty of the present paper. In addition
  \cite{Chen:Tang1,Chen:Tang2} consider  evolution equations involving non-local $p$-Laplacian type operators both in periodic and random media.}

\subsection*{About the method of proof. Organization of the paper}

Concerning the organization of the paper, in Sections
\ref{sec2.1}--\ref{sec2.2}  we introduce the basic notions used
throughout the paper. In Section \ref{sec2.2a} we recall some  facts
about stable-like processes.
To show our main result, formulated in Theorem \ref{th.intr1}, we embed the space  $\Theta$ into $\Om$
- the set of all continuous and bounded matrix valued functions, by
assigning to each $\theta\in  \Theta$ its realization, i.e.\ $J(\theta)(x):= a(x;
\theta),$ $x\in\bbR^d$. The space $\Om$ is Polish, when considered
with the standard Fr\'echet metric. We prove,  {see Theorem
  \ref{thm.CLT}}, that the conclusion of our main result holds under
the  {additional} assumption that an $\Om$-valued {\em environment
  process}, which describes the random environment  $\om\in \Om$ from
the vantage point of the trajectory, has an invariant ergodic
probability measure $\mu_*$. This measure is equivalent with $\mu$ - the
push-forward measure of $m$ by $J$. The density
$\Phi_*=\frac{d\mu_*}{d\mu}\circ J$ appears in the statement of
Theorem \ref{th.intr1}. Such environment process is rigorously
introduced in Section \ref{sec2.4}.
In Section
\ref{Sec.CLT}, we formulate  and prove  the
homogenization result, provided that we know that the environment
process possesses an invariant and ergodic measure equivalent with
$\mu$, the law of the random environment. Section \ref{sec3} is
devoted to showing the existence and uniqueness of such a measure, with a
strictly positive  invariant density $\Phi_*$, see
Theorem \ref{thm:main}.  The proof uses  a weak
form  of  the
Alexandrov-Bakelman-Pucci (ABP) estimates for solutions of the
Dirichlet problem for equations involving a 
non-local operator of the form \eqref{Lom}. Such estimates follow from
the results of \cite{ABP}, see Section \ref{sec-abp}. Using this 
we can prove that the invariant density exists and  is in fact $L^{1+\delta}$
integrable for  $0<\delta<\frac{\al}{2d-\al}$.

The proof of Theorem \ref{th.intr1} is presented in Section \ref{sec4.5}
under an additional assumption that the uniform ellipticity condition
\eqref{eq.uel} holds. In Section
\ref{sec.main.gen} we relax this assumption and prove the result in full generality. 

Finally in the Appendix, we show some additional facts, which are
needed in our paper. As we have already mentioned in the foregoing, Section \ref{sec-abp} formulates a version of the
Alexandrov-Bakelman-Pucci estimates  {for integro-differential operators}, see Theorem \ref{ABP}, that can
be inferred from \cite[Theorem 1.3]{ABP}. In Section
\ref{appC} we prove  a result, see Proposition \ref{prop012812-23},  about irreducibility of stable-like processes.
 Section \ref{appC1} is devoted to derivation of the formula for the Fourier symbol of an isotropic stable-like process.
Section \ref{appD} contains the proof of the fact that the set of
$d\times d$, non-negative symmetric
 matrix valued functions, such that there exists a unique Feller
 process corresponding to the respective Fourier symbol \eqref{beq01},
 is a Borel subset of $\Om$.

\subsection*{Acknowledgments}  
T. Komorowski acknowledges the support of the NCN grant 2020/37/B/ST1/00426,
T. Klimsiak acknowledges the support of the NCN grant   2022/45/B/ST1/01095.

\section{General setup and notations}
\label{Sec:prob_setup_result}

\subsection{Some generalities}
\label{sec2.1}
We denote by $B(y,r)$ the open  ball of radius $r>0$ centered at $y\in \bbR^d$, with respect to the Euclidean metric on $\bbR^d$ and  by $\mathcal{L}_d$ the Lebesgue measure on $\bbR^d$. Let $B_b(\bbR^d)$ (resp.\ $C_b(\bbR^d)$) be the space of all Borel measurable (resp.\ continuous) and bounded functions on $\bbR^d$. We denote the supremum norm  of any bounded function $f$ on $\bbR^d$ by    $\|f\|_\infty:=\sup_{x\in\bbR^d}|f(x)|$. 
Given $k\in \bbN$, let $C_b^k(\bbR^d)$ be the space of all $k$-times
differentiable functions with continuous and bounded derivatives. For
any  $f\in C^k_b(\bbR^d)$, we denote by $\|f\|_{k,\infty}$ the norm
defined as the sum of the supremum norms of the function and all its
derivatives up to order $k$ included.  {For any $f\in C^1_b(\bbR^d)$,
  we denote its gradient  by $\nabla f$. We also consider the space $C^\infty_c(\bbR^d)$  of all the compactly supported, smooth functions on $\bbR^d$.}  For a given $\la>0$, we denote by ${\cal  S}_d^+(\la)$ the set of all matrices $a\in{\cal S}_d^+$ such that
$a-\la {\rm I_d}\in {\cal S}_d^+$, where ${\rm I_d}$ is the $d\times d$
identity matrix.

\subsection{Probability  space with a group of measure preserving  transformations}
\label{sec2.2}

 {Let  $\Omega$ be the space of all functions $\om:\bbR^d\to \mathcal{S}_d^+$, which are   continuous and satisfy}
\[\|\om(x)\|\le \Lambda,\qquad 
\operatorname{Tr}\left(\om(x)\right) \ge \lambda,\qquad  x\in\bbR^d,\]
{where $\Lambda,\lambda>0$ are the same constants appearing in \eqref{eq.MH1}.}
The space   is Polish when equipped with the Fr\'echet metric 
\[{\rm d}  (\om_1, \om_2) \, := 
\,
\sum_{K=1}^{+\infty}\frac{1}{2^K}\cdot\frac{\|\om_1-\om_2\|_{\infty,K}}{1
  +\|\om_1-\om_2\|_{\infty,K}},\]
where, for a given $K>0$,  
\[
\|\om\|_{\infty,K}:=
 \sup_{|x|\le K }
 |\om(x)|, \qquad \om \in \Omega.
\]
For ${\cal r}>0$,  {we introduce the subspace $\Omega_{\cal r}$ of $\Omega$ which consists of all functions $\omega\colon \bbR^d\to{\cal S}_d^+({\cal r})$ such  that both $\omega$ and $\nabla \omega$ are continuous.} The space $\Omega_{\cal r}$ is Polish when equipped  with the Fr\'echet metric
\[{\rm d}_{\cal r} (\om_1, \om_2) \, := 
\,
\sum_{K=1}^{+\infty}\frac{1}{2^K}\cdot\frac{\|\om_1-\om_2\|_{1,\infty,K}}{1
  +\|\om_1-\om_2\|_{1,\infty,K}}.\]
Here
\[
\|\om\|_{1,\infty,K}:=
 \sup_{|x|\le K }
 |\om(x)|+ \sup_{|x|\le K }
 |\nabla \om(x)|,\qquad \om\in \Omega_{\cal r}.
\]
We denote by $\mathcal B(\Omega_{\cal r})$ and  $\mathcal B(\Omega)$ the Borel  $\sigma$-fields
of $(\Omega_{\cal r},{\rm d}_{\cal r})$ and $(\Omega,{\rm d})$, respectively.
{Moreover, let $B_b(\Omega)$ and $B_b(\Omega_{\cal r})$ (resp.\ $C_b(\Omega)$ and $C_b(\Omega_{\cal r})$) be the
spaces of all Borel measurable (resp.\ continuous) and bounded functions on
 $(\Omega,{\rm d})$ and $(\Omega_{\cal r},{\rm d}_{\cal r})$, respectively.} 
Note that by  \cite[Theorem 8.3.7]{cohn},
$\Omega_{{\cal r}}$ is a Borel measurable subset of $\Omega$, hence $\mathcal B(\Omega_{\cal r})\subset \mathcal B(\Omega)$.

{Consider an additive group of
transformations $(\tau_x)_{x\in\bbR^d}$ acting on $\Omega$ as follows} 
\[\tau_x\om(y):=\om(x+y),  \qquad y\in\bbR^d.\]
Clearly, $\tau_x(\Omega_{\cal r})\subset \Omega_{\cal r}$. Note that for any $f\in C_b(\Omega)$, we have
\[\lim_{x\to0}f\big(\tau_x\omega)=f(\omega),\qquad\om\in \Omega.\]

Given two measure spaces $(\Sigma_i,{\cal
  A}_i,m_i)$, $i=1,2$ and a measurable mapping $S:\Sigma_1 \to\Sigma_2$, 
  we denote by $S_{\sharp} m_1$ the push-forward of $m_1$
through $S$, i.e.\ the measure on $(\Sigma_2,{\cal
  A}_2)$ given by
$S_{\sharp} m_1(A)=m_1\big(S^{-1}(A)\big)$ for any $ A$ in $ \mathcal A_2$. We introduce the mapping $J:\Theta\to \Omega$ by letting
\[
J(\theta)(x):= a(x;\theta),\qquad x\in\bbR^d.\]
{By the assumptions made in Hypothesis \ref{H1}, the function}
$J\colon(\Theta,\mathcal H)\to (\Omega,\mathcal B(\Omega))$ is
measurable. Let  $\mu:= J_{\sharp}m$ be a measure  on $(\Omega,\mathcal B(\Omega))$.
By stationarity and ergodicity  of $(a(x))_{x\in\bbR^d}$ the measure
$\mu$ is invariantand ergodic under the action of the group, i.e.
\begin{equation}
\label{eq.inv}
(\tau_x)_{\sharp}\mu=\mu,\qquad x\in\bbR^d
\end{equation}
and    if  $A\in \mathcal B(\Omega)$ satisfies $\tau_xA=A$ for all $x\in \bbR^d$, then $\mu(A)=0$ or $1$.

\subsection{Random stable-like processes}
\label{sec2.2a}
{Recall that   ${\mathcal D}$ is} the
space of all  c\`adl\`ag paths $\zeta\colon [0,+\infty)\to \bbR^d$, equipped with
the $J_1$-Skorokhod topology. {We now introduce the canonical process $(X_t)_{t\ge 0}$ by letting}
\begin{equation}
\label{eq:can}
X_t(\zeta):=\zeta(t), \qquad \zeta \in {\mathcal D},
\end{equation}
and its natural filtration $({\mathcal F}_t)_{t\ge0 }$ by ${\mathcal F}_t:=\sigma\left( X_s,\,0\le s\le t\right)$. Then, ${\cal F}:=\sigma\left( X_t,\,t\ge 0\right)$ is the Borel $\si$-algebra on ${\cal D}$.

Given the function ${\mathbf a}: \Omega\to \mathcal{S}_d^+$ defined by
${\mathbf a}(\om):=\om(0)$, consider  the associated (random) Fourier
symbol  
\begin{equation}
  \label{beq01.reg}
  \mathbf q(\xi;\omega) := \int_{\mathbb
    R^d}(1 -e^{iz\cdot\xi}+i z \cdot
  \xi\mathds{1}_{\{|z|\le1\}})\frac{\langle  {\mathbf
      a}(\omega)z,z\rangle}{|z|^{d+\alpha+2}}\,dz=\int_{\mathbb
    R^d}(1 -\cos(z\cdot\xi))\frac{\langle  {\mathbf a}(\omega)z,z\rangle}{|z|^{d+\alpha+2}}dz,
\end{equation}
where $(\xi,\omega)\in\bbR^d\times\Om$,
and the corresponding operator $\mathbf q(D;\omega)$ on $C_c^2(\bbR^d)$ by
\begin{equation}
  \label{qxd}
\mathbf q(D;\omega)u(x):=\int_{\bbR^d}\mathbf q(\xi;\tau_x\omega)\hat u(\xi) e^{ix\cdot\xi}\,d\xi, \qquad x\in\bbR^d.
\end{equation}
Furthermore, one can conclude, see Appendix \ref{appC1}, that
\begin{equation}
\label{eq.bb1}
\mathbf q(\xi;\omega)= C_{d,\al}{\rm Tr}\Big({\mathbf
      a}(\omega)\Big)|\xi|^\al+c_{d,\al} \langle  {\mathbf
      a}(\omega)\hat\xi,\hat\xi\rangle|\xi|^\al,
\end{equation}
where $\hat \xi:=\xi/|\xi|$ and the constants $C_{d,\al}, c_{d,\al}>0$
depend only  the dimension $d$ and  exponent $\al$. 

We say that the martingale problem
corresponding to  $\mathbf q(D;\omega)$, $\om\in\Om$, is well-posed, cf.\ \cite{stroock}, if for
every Borel probability measure
$\nu$ on $\bbR^d$, there exists a unique
Borel  probability measure $\bbP^{\nu;\om}$ on $ {\cal D} $, called a {\em solution to the martingale problem}  
for $\mathbf q(D;\omega)$ with initial distribution $\nu$, such that
\begin{itemize}
\item[i)] $\bbP^{\nu;\om}\left(X_0\in A\right)=\nu(A)$ for any Borel measurable
 $A\subset \bbR^d$,
\item[ii)] for any $f\in C_c^\infty(\bbR^d)$,
 the process
\begin{equation}
\label{Mtf}
M^\om_t[f]:= f(X_t)-f(X_0)-\int_0^t [\mathbf q(D;\omega) f](X_r)\,dr,\qquad t\ge 0
\end{equation}
is a  (c\`{a}dl\`ag paths) $({\cal F}_t)_{t\ge0}$-martingale under measure $\bbP^{\nu;\om}$.
\end{itemize}
As usual, we write $\bbP^{x;\om}:=\bbP^{\delta_x;\om}$,  $x\in\bbR^d$. The expectations   
with respect to $\bbP^{\nu;\om}$ and $\bbP^{x;\om}$ shall be denoted 
by $\bbE^{\nu;\om}$ and $\bbE^{x;\om}$, respectively.

{Let $\Omega^{\cal M}$ be the set of all $\omega\in \Omega$ such that there exists a unique Feller process corresponding to the Fourier symbol $\mathbf q(\xi;\omega)$.  One can show, see Appendix \ref{appD}, that $\Omega^{\cal M}\in \mathcal{B}(\Omega)$.}
Moreover, by the assumptions made in Hypothesis \ref{H1}, the measure $\mu:= J_{\sharp}m$ is supported in $\Omega^{\mathcal M}$. 
By \cite[Theorem 1.3]{BKS}, it also follows that  $\Omega_{\cal r}\subset  \Omega^{\mathcal M}$
for any $\cal r>0$.

{Theorem $1.1$ and Lemma $2.1$ in \cite{kuhn} imply that}
for each $\omega\in \Om^{\mathcal M}$, the operator
$\mathbf q(D;\omega)$ defined on $C_c^2(\bbR^d)$ is closable in $C_0(\bbR^d)$ and its
closure generates the Feller transition probability semigroup
$(P^\omega_t)_{t\ge 0}$, satisfying
\begin{equation}
  \label{032310-23}
P^\omega_tf(x)=\bbE^{x;\om} f\big(X_t\big),\qquad f\in
C_0(\bbR^d),\,x\in\bbR^d,\,t\ge0.
\end{equation}
For any $\beta>0$ and $\om\in \Omega^{\mathcal M}$, we introduce the $\beta$-resolvent operator
  $R_\beta^\omega\colon { B_b(\bbR^d)}\to { B_b(\bbR^d)}$ as
  \begin{equation}
    \label{012812-23}
R_\beta^\omega f(x):=\int_0^{\infty}e^{-\beta t}P^\omega_tf(x)\, { dt},\qquad
x\in  \bbR^d.
\end{equation}
Since the Fourier symbol of $(\bbP^{x;\omega})_{x\in\bbR^d}$ is given by 
$$
q(x,\xi;\omega):= \mathbf q(\xi,\tau_x\omega),  \quad (x,\xi,\om)\in \mathbb
R^{2d}\times  \Omega^{\mathcal M},
$$
with $\mathbf q$  {defined in} \eqref{beq01.reg},
one can show that $\tau_x(\Omega^{\mathcal M})\subset \Omega^{\mathcal M}$ and 
\begin{equation}
\label{022905-20}
\mathbb{P}^{x+y;\omega}=(s_y)_{\sharp}\mathbb{P}^{x;\tau_y\omega},\quad x,y\in\bbR^d,\om \in \Omega^{\mathcal M},
\end{equation}
where   $s_y:{\cal D}\to{\cal D}$ is given by
$s_y(\zeta)(t):=y+\zeta(t)$, $t\ge0$.




For each $\eps\in(0,1)$ and $\om\in \Omega^{\mathcal M}$, 
we now consider the scaled process
$\left(\bbP_{\eps}^{x;\om}\right)_{x\in\bbR^d}$ such that $\bbP_{\eps}^{x;\om}(X_0=x)=1$ and
whose Fourier symbol $q_\eps\colon \bbR^{d}\times \bbR^{d}\to \bbR^d$ equals
\begin{equation}
  \label{beq01eps}
    q_\eps(x,\xi;\omega):=\frac{1}{\varepsilon^{\alpha}}q\Big(\frac{x}{\varepsilon},\varepsilon\xi;\omega\Big)=C_{d,\al}{\rm Tr}\Big({\mathbf
      a}(\tau_{x/\eps}\omega)\Big)|\xi|^\al+c_{d,\al} \langle  {\mathbf
      a}(\tau_{x/\eps}\omega)\hat\xi,\hat\xi\rangle|\xi|^\al.
\end{equation}
Denote by ${\cal T}_\eps:{\cal D}\to{\cal D}$ the mapping ${\cal T}_\eps(\zeta)(t):=\eps
 \zeta(t/\eps^{\al})$. Note that
\begin{equation}
  \label{011903-24}
  \bbP_{\eps}^{x;\om} =\big({\cal
 T}_\eps\big)_{\sharp}\bbP^{x/\eps;\om}.
\end{equation}
Let $\bbE_{\eps}^{x;\om} $ be the expectation with respect to the measure $\bbP_{\eps}^{x;\om}$. We denote by $(P^{\omega}_{t,\varepsilon})_{t\ge0}$  the Feller semigroup associated with the process $(\bbP^{x;\omega}_\varepsilon)_{x\in\bbR^d}$ and by $\mathbf q_\varepsilon(D;\omega)$  the corresponding generator.

Given a Borel probability measure $\nu$ on $\Omega$ that is supported
in $\Omega^{\mathcal M}$,  {we introduce the measure $\bbP^{x;\nu}_\eps$ on  $(\Omega\times \mathcal{D},\mathcal{B}(\Omega)\otimes {\cal F})$}
\begin{equation}
  \label{Qeps1234}
\bbP^{x;\nu}_\eps (A)=\int_{\Omega}\nu(d\om)\int_{\cal D}1_A(\om,\zeta) \bbP^{x;\om}_\eps
(d\zeta),\qquad A\in { \mathcal{B}(\Omega)}\otimes {\cal F}
\end{equation}
and  let $\bbE^{x;\nu}_\eps$ be the respective expectation.
To lighten somewhat the notation, we omit writing the superscipt when
$x=0$ and subscript when $\eps=1$  in the notation of measures and
their respective expectations, e.g. we write $\bbE^{\nu} $ and $\bbP^{\nu}$ when $x=0$ and $\eps=1$.



\subsection{The random environment as seen from the particle}
\label{sec2.4}

For each $\omega\in \Omega^{\mathcal M}$, we introduce an $\Omega^{\mathcal M}$-valued process $(\eta_t)_{t\ge0}$
over the probability space $({\cal D},{\cal F}, \mathbb
P^{\omega})$, sometimes referred to as the environment process,
defined by
\begin{equation}
\label{eta-t}
\eta_t(\omega) \, := \, \tau_{X_t}\omega,\qquad t\ge0.
\end{equation}

\begin{proposition}
 For each $\omega\in \Omega^{\mathcal M}$, the process $(\eta_t(\omega))_{t\ge 0}$ is
$\left(\mathcal{F}_t\right)_{t\ge0}$-Markovian under measure $\bbP^{\omega}$. Its transition semigroup $({\frak P}_t)_{t\ge 0}$ is
given by
\begin{equation}
  \label{Pt}
{\frak P}_tF(\omega)\, = \bbE^{\om} F\big(\eta_t(\omega)\big)=P_t^\om\tilde F(\cdot\,;\om)(0), \qquad F\in B_b(\Omega^{\mathcal M}),\,\omega\in \Omega^{\mathcal M},
\end{equation}
where  $\tilde F(y;\om):=F(\tau_y\om)$.
\end{proposition}
The proof of the above result can be obtained following the same arguments as in \cite[Proposition 9.7]{KLO12}.
We may extend ${\frak P}_t$ to  {$B_b(\Omega)$} by letting 
\[
{\frak P}_t F(\omega):= F(\omega),\quad \omega\in \Omega\setminus \Omega^{\mathcal M}.
\]
Observe that if $\omega\in \Omega_{\cal r}$ (resp.\ $\omega\in \Omega^{\mathcal M}$), then $\eta_t(\omega)\in \Omega_{\cal r}$
(resp.\ $\eta_t(\omega)\in \Omega^{\mathcal M}$) for any $t>0$.
Thus, $\eta_t$ may be regarded as a Markov process on  either $\Omega_{\cal r}$, or $\Omega^{\mathcal M}$. 
Using Markov processes theory nomenclature both $\Omega_{\cal r}$ and $\Omega^{\cal M}$ are {\em absorbing sets}
(see \cite[Definition 12.27]{Sharpe} and also \cite[Theorem
12.30]{Sharpe}).

\subsection{The  homogenization result}
\label{Sec.CLT}

\begin{thm}[Quenched  invariance principle]
\label{thm.CLT}
Let $x\in \bbR^d$.
Assume that there exists an invariant ergodic probability measure $\mu_*$
for the process $(\eta_t)_{t\ge0}$, which is equivalent to $\mu$.
Then, as
$\eps\downarrow0$, the measures 
$\left(\bbP^{x;\om}_\eps \right)_{\eps\in(0,1)}$ weakly converge in $\mathcal{D}$, $\mu$-a.s.\ in $\om$, to $\bar \bbQ^x$, the law of an $\alpha$-stable process $\left\{x+\bar
  Z(t)\right\}_{t\ge0}$ with the L\'evy symbol
\[\bar q(\xi)=\int_\Omega  {\mathbf q} (\xi;\omega)\,\mu_*(d\omega) =C_{d,\al}{\rm Tr}(\bar {\mathbf a} )|\xi|^\al+c_{d,\al} \langle  \bar{\mathbf a} \hat\xi,\hat\xi\rangle|\xi|^\al,\]
where $\bar{\mathbf a}=\int_\Omega  {\mathbf a} (\omega)\,\mu_*(d\omega)$.
The above means that  for $\mu$-a.s.  $\om$ 
\[
  \lim_{\eps\downarrow0} \bbE_{\eps}^{x;\om} f=\bar \bbE^x
    f
  \]
  holds for any   $f$ - a bounded and
  continuous function on ${\cal D}$. Here, $\bar \bbE^x$ is the expectation with respect to $\bar \bbQ^x$.
In particular, $\mu$-a.s. $\om$, we have 
\[
\lim_{\eps\downarrow0} P^\omega_{t,\eps}f(x)=\bar P_tf(x) ,\qquad f\in
C_b(\bbR^d),\,t>0,\,  x\in\bbR^d.
\]
Here, $\left(\bar{ P}_t\right)_{t\ge0}$ is the transition probability semigroup of $\left(\bar Z(t)\right)_{t\ge0}$.

\end{thm}

\begin{proof}
The family $(\mathbb
P^{x;\om}_\varepsilon)_{\varepsilon\in (0,1]}$ is tight for any $\om\in \Omega$.
Indeed, there exists  a constant $C>0$ such that for any $u\in C^2_b(\bbR^d)$, $\varepsilon\in (0,1)$ and $\om\in \Omega$,
\begin{equation}
\label{012201-24}
\|q(\cdot,D;\omega)u(\cdot)\|_\infty\le C\|u\|_{2,\infty} .
\end{equation}
Hence, tightness of $(\bbP^{x;\om}_\varepsilon)_{\varepsilon\in
  (0,1]}$ follows, see the proof of  \cite[Theorem 4.9.2]{Kol}, or
\cite[Theorem A.1]{stroock}.

To finish the proof of the theorem, it suffices therefore to prove the
convergence of finite dimensional distributions. In what follows we
show that for any $t>0$, $f\in C_b(\bbR^d)$,  
\begin{equation}
  \label{031803-24}
  \lim_{\eps\downarrow0}  \bbE_{\eps}^{x;\om} f(X_t)=\bar \bbE^x f(X_t), \qquad \text{$\mu$-a.s.\ in $\om$}.
\end{equation}
The generalization of \eqref{031803-24} to the case of a finite number of times is straightforward. Since $\mu$ is invariant under $(\tau_x)_{x\in\bbR^d}$, it is enough to consider only the case when $x=0$, thanks to \eqref{022905-20}.

According to \cite[Theorem 1.1]{stroock}, for each $\om\in \Omega^{\mathcal M}$, $\xi\in\bbR$ and
$\eps>0$, the process
$$
\exp\left\{iX_t\cdot\xi-i x\cdot\xi-\int_0^tq_\eps(X_s,\xi;\om)ds\right\},\quad t\ge0
$$
is a mean one, $({\cal F}_t)$-martingale under the measure $\mathbb
P^{0;\om}_\varepsilon$. Taking into account the definition of the
scaled path measures,  cf.\ \eqref{011903-24}, in order to prove \eqref{031803-24}
  it suffices to only
show that for any $t>0$, $\xi\in\bbR^d$,
\begin{equation}
  \label{073012-23}
\lim_{\eps \downarrow 0} \int_0^{t/\eps^{\al}}
  { q\left(X_s,\varepsilon\xi\right)}\,ds= t \bar
    q (\xi ),\quad  \mathbb
P^{\mu}\mbox{-a.s.} 
  \end{equation}
Using the environment process $(\eta_t)_{t\ge 0}$ defined in
\eqref{eta-t} and formula \eqref{eq.bb1}, we see that
\eqref{073012-23}
is equivalent with proving that
\[\lim_{\eps \downarrow 0} \eps^{\al}\int_0^{t/ \varepsilon^{\alpha}} \mathbf
q( \xi;\eta_s)\,ds= t \bar
    q (\xi ),\quad  \mathbb
P^{\mu}\mbox{-a.s.} \]
By the individual Birkhoff ergodic theorem the above convergence holds $\mathbb
P^{\mu_*}$-a.s. The conclusion of the theorem is then a consequence of
the fact that  $\mu_*$ and $\mu$ are equivalent.
\end{proof}

{In light of the above result,} the proof of the main Theorem \ref{th.intr1} is concluded once we show the existence of a probability measure $\mu_*$ as in the statement of Theorem \ref{thm.CLT}. This is going to be the main objective of  Section \ref{sec3}.

\section{On the existence of an ergodic invariant measure for the environment process}
\label{sec3}

The main purpose of the present section is to prove the following result.
\begin{thm}
  \label{thm:main}
Let $1\le p<d/(d-\al/2)$. 
Then, there exists $\Phi_*\in L^{p}(\mu)$ such that:
\begin{enumerate}
\item $\|\Phi_*\|_{L^1(\mu)}=1$;
\item the measure $\mu_*$ on $(\Omega,\mathcal{B}(\Omega))$ given by
  $d\mu_\ast=\Phi_* d\mu$ 
is invariant  under the dynamics of the environment process $(\eta_t)_{t\ge 0}$, i.e.
\[
\int_{\Omega}{\frak P}_tF \, d\mu_\ast = \int_{\Omega}F \, d\mu_\ast
  ,\qquad F\in B_b(\Omega),\, t\ge0;
  \]
\item  
 $\Phi_*>0$ $\mu$-a.s.\ on $\Omega$ and  in consequence
 the measure $\mu_*$ is ergodic, i.e.   if $F\in L^\infty(\mu_\ast)$ satisfies ${\frak P}_tF=F$ 
 for any $t>0$, then $F$ is constant $\mu_\ast$-a.s.\ in $\Omega$.
\end{enumerate}
\end{thm}

The proof of the theorem is presented in  Sections
\ref{sec4} -- \ref{sec.main.gen} below.
 In  Sections
\ref{sec4} -- \ref{sec4.5} we prove Theorem \ref{thm:main} under some
additional assumptions about the non-degeneracy and  regularity  of the random field. Namely, we suppose that the
uniform ellipticity condition  \eqref{eq.uel}
holds  and the realizations of the random field $\left(a(x;\omega)\right)_{x\in\bbR^d}$
are  deterministically bounded together with their  first derivatives.  Finally, in Section \ref{sec.main.gen} we finish the proof of the theorem by discarding these additional regularity assumptions.

\subsection{An ergodic theorem}
\label{sec4}
Our first result is a version of the ergodic theorem somewhat
analogous to the one that can be found in    \cite{part}, see also
\cite[Lemma 3.2]{PV2}. Before its
formulation we need some notation. Let ${\rm e}_i$, $ i=1,\ldots,d$ be the canonical basis of $\bbR^d$. By the one dimensional
unit torus $\bbT$ we understand the interval $[-1/2,1/2]$ whose
endpoints  are identified. Given $M>0$, we can then denote by
$\bbT^d_{M}:=(M\bbT)^d$ the $d$-dimensional torus of length $M$. 
Furthermore, we let  
$Q_M \, := \, \left[- M/2,  M/2\right]^d$.  We shall
also denote by
$\ell_M(dy):=M^{-d}dy$ the normalized Lebesgue measure both on
$\bbT^d_{M}$ or $Q_M$.  {Let $B_b(\bbT^d_{M})$ (resp.\ $C(\bbT^d_{M})$) be the space of all bounded Borel measurable (resp.\ continuous) functions on $\bbT^d_{M}$.
Given $k\in \bbN$, let $C^k(\bbT^d_{M})$ be the space of all $k$-times differentiable functions on $\bbT^d_{M}$ with continuous derivatives.}

\begin{lemma}
\label{lm023005-20}
There exists $\bar\omega\in \Omega_{\cal r}$  such that
the sequence $(\bar\mu_M)_{M\in \bbN}$ of measures on $(\Omega_{\cal r},{\cal B}(\Omega_{\cal r}))$, given by
\[
\bar\mu_M(A)\, :=\, \int_{Q_M}1_A(\tau_y\bar\omega)\, \ell_M(dy),
\qquad  A\in {\cal B}(\Omega_{\cal r}),\]
weakly converges, as $M\to\infty$, to $\mu$.
\end{lemma}
\begin{proof}
First, we observe that there
exist
a metric $\bar{\rm d}$ on
$\Omega_{\cal r}$ which is equivalent with  ${\rm d}_{\cal r}$ and
a countable family $\mathcal Z:=(F_n)_{n\in \bbN}$ of bounded
functions on $\Omega_{\cal r}$ such that $\mathcal{Z}$ is dense in $U_{\bar{\rm
    d}}(\Omega_{\cal r})$ in the supremum norm.
Here $U_{\bar {\rm d}}(\Omega_{\cal r})$ is the space of all  real valued, uniformly
continuous in metric $\bar {\rm d} $ functions on $\Omega_{\cal r}$.
This can be seen as follows. Since $\Omega_{\cal r}$ is Polish, it is
well-known that $\Omega_{\cal r}$
is homeomorphic to a subset of a compact metric space - the Hilbert
cube. Therefore, there exists an equivalent metric $\bar{\rm d}$ on
$\Omega_{\cal r}$ such that $\left(\Omega_{\cal r}, \bar{\rm d}\right)$ is totally
bounded. It then follows that the completion of $\Omega_{\cal r}$ under
$\bar{\rm d}$, denoted by $\bar \Omega_{\cal r}$, is compact. The space
$U_{\bar{\rm d}}(\Omega_{\cal r})$ is isometrically isomorphic with the space $C(\bar \Omega_{\cal r})$
of continuous functions on $\bar \Omega_{\cal r}$, both equipped with the topology
of uniform convergence. Since it is known (cf. \cite[Lemma
VI.$8.4$]{dunford-schwarz}) that the latter is separable, so is
$U_{\bar{\rm d}}(\Omega_{\cal r})$.

By the individual ergodic theorem,
we can choose $\tilde\Omega_{\cal r}\subset \Omega_{\cal r}$ such that $\mu(\tilde\Omega_{\cal r})=1$ and for any
$\bar\omega\in \tilde\Omega_{\cal r}$ and   $F\in{\mathcal Z}$
\begin{equation}
  \label{011307-23}
\lim_{M\to\infty}\int_{Q_M}F(\tau_y\bar\omega)\, \ell_M(dy)= 
\lim_{M\to\infty}\frac{1}{M^d}\int_{Q_M}F(\tau_y\bar\omega)\, dy \,= \, \int_{\Omega_{\cal r}}F\,d\mu.
\end{equation}
A density argument implies that \eqref{011307-23} holds for any function $F$ in $U_{\bar{\rm d}}(\Omega_{\cal r})$
and $\bar \omega\in  \tilde\Omega_{\cal r}$.
 The conclusion of the lemma follows
then from  \cite[Theorem $1.1.1$]{stroock-varadhan}.
\end{proof}

We can now state our version of the ergodic theorem.

\begin{thm}
  \label{thm:parta}
There exist a sequence $(\omega_n)_{n\in \bbN}$ in $\Omega_{\cal r}$ and an increasing 
sequence of positive numbers  $(M_n)_{n\in \bbN} $ such that $M_n\to \infty$ and
\begin{enumerate}
\item each $\omega_n$ is $M_n$-periodic in each variable, i.e.\
  $\tau_{M_n{\rm e}_i}\omega_n=\omega_n$, for $n\in \bbN$ and   $i=1,\ldots,d$;
\item the following sequence of probability measures on
  $(\Omega_{\cal r},{\cal B}(\Omega_{\cal r}))$
\begin{equation}
\label{062905-20}
\mu_n(A) \, := \, \int_{\bbT^d_{M_n}}\mathds{1}_A(\tau_y\omega_n)\, \ell_n(dy), \qquad
A\in {\cal B}(\Omega_{\cal r}),
\end{equation}
weakly converges to $\mu$, as $n\to\infty$, i.e.
\begin{equation}
\label{072905-20}
\lim_{n\to\infty}\int_{\Omega_{\cal r}} F
\, d\mu_n= \int_{\Omega_{\cal r}} F  \,
d\mu , \qquad F\in C_b(\Omega_{\cal r}).
\end{equation}
\end{enumerate}
\end{thm}
\proof Thanks to Lemma \ref{lm023005-20},
there exists $\bar\om\in \Omega_{\cal r}$ such that
\begin{equation}
\label{062905-20M31}
\lim_{M\to\infty}\int_{Q_M}F(\tau_y\bar\om)\,
\ell_M(dy)=\int_{\Omega_{\cal r}}F\, d\mu, \qquad F\in C_b(\Omega_{\cal r}).
\end{equation}
Fix an arbitrary $\rho>0$
and $F\in U_{{\rm d}_{\cal r}}(\Omega_{\cal r})$. 
Let us   consider  increasing  sequences of integers
$(M_n')_{n\in\bbN}$, $(M_n)_{n\ge1}$ and $(\om_n)_{n\in\bbN}\subset \Omega_{\cal r}$
satisfying:
  \begin{itemize}
    \item[-] each $\om_n$ is
      $M_n$-periodic in every direction of the variable $x$;
      \item[-] $M_{n}'<M_n$, $n=1,2,\ldots$ and $\lim_{n\to\infty}\frac{M_{n}'}{M_n}=1$;
    \item[-] 
    {$\lim_{n\to\infty}\sup_{y\in Q_{M_{n'}}}{\rm d}_{\cal r}({  \tau_y\om_n, \tau_y\bar\om})=0$}.
    \end{itemize}
Thanks to \eqref{062905-20M31}, we have
\begin{equation}
\label{062905-20M3}
\lim_{n\to\infty}\left|\int_{Q_{M_n}}F(\tau_y\bar\om)\, \ell_{M_n}(dy)-\int_{\Omega_{\cal r}}F\,
  d\mu\right|=0.
\end{equation}
Since $F\in U_{{\rm d}_{\cal r}}(\Omega_{\cal r})$, there exists $n_0$ such that 
\begin{equation}
\label{013005-20}
|F(\tau_y\om_n)-F(\tau_y\bar\om)|\,< \,\rho,\qquad
y\in Q_{M_n'},\,n\ge n_0.
\end{equation}
Recalling the definition of the measures $\mu_n$ in \eqref{062905-20}, we infer that
\begin{multline}
  \label{010602-24}
\left|\int_{\Omega_{\cal r}}F\, d\mu_n-\int_{\Omega_{\cal r}}F
\, d\mu\right| = \left|\int_{\bbT^d_{M_n}}F(\tau_y\om_n)\, \ell_{M_n}(dy)-\int_{\Omega_{\cal r}}F
\, d\mu\right|\\
\le M_n^{-d}\int_{Q_{M_n}}\left|F(\tau_y\om_n)-F(\tau_y\bar\om)\right|\, dy+\left|\int_{Q_{M_n}}F(\tau_y\bar\om)\, \ell_{M_n}(dy)
-\int_{\Omega_{\cal r}}F
  d\mu\right|
\end{multline}
The second expression on the utmost right hand-side tends to $0$, as $n\to\infty$, thanks to
\eqref{062905-20M3}.
The first one on the other hand can be estimated from
\eqref{013005-20}
by
\[
\rho +\frac{1}{M^d_n}\int_{Q_{M_n}\smallsetminus Q_{M_n'}}\left|F(\tau_y\om_n)-F(\tau_y\om)\right|\, dy\le \,
\rho+ 2\|F\|_\infty\left[1-\left(\frac{M_n'}{M_n}\right)^d\right].
\]
Since $\frac{M_n'}{M_n}\to1$, as $n\to\infty$, the
above estimate together with \eqref{010602-24}
imply that
\[
\limsup_{n\to\infty}\left|\int_{\Omega_{\cal r}}F\, d\mu_n-\int_{\Omega_{\cal r}}F
\, d\mu\right| \le \rho
\]
for any arbitrary $\rho>0$, which in turn implies \eqref{072905-20}
for any $F\in U_{{\rm
    d}}(\Omega_{\cal r})$.  Finally,  another  application of  \cite[Theorem
$1.1.1$]{stroock-varadhan} allows us to conclude the proof of Theorem
\ref{thm:parta}.\qed

\subsection{Construction of an invariant density for the environment process}
\label{sec4.4}
 
Let $(\om_n)_{n\in\bbN}$ and $(M_n)_{n\in\bbN}$ be as in the statement of Theorem
\ref{thm:parta}. {For each $n\in\bbN$, we associate with  $\om_n\in \Omega_{\cal r}$} the
operator $\mathbf q(D;\omega_n)\colon C^2_c(\bbR^d)\to C_0(\bbR^d)$ as in \eqref{qxd} and the corresponding Feller process $(\mathbb{P}^{x;\omega_n})_{x\in\bbR^d}$.
Let $\pi_{n}: \bbR^d\to \bbT^d_{M_n}$ be the canonical projection of
$\bbR^d$ onto $\bbT^d_{M_n}$. Then, the process
$\left(\tilde{\mathbb{P}}^{x;n}\right)_{x\in\bbT^d_{M_n}}$ defined by
$\tilde{\mathbb{P}}^{x;n}:=(\pi_{n})_\sharp\mathbb{P}^{x;\omega_n}$ is
strongly Markovian and Feller, with the transition probability densities given by
 \[\tilde p_{ t}^n(x,y):= \sum_{m\in\bbZ^d}  p_{t}^{\om_n}(x,y+M_nm),\qquad t\ge0, \, x,y\in \bbT^d_{M_n}.
\]
Here, $p_{ t}^{\om_n}(x,y)$ are the transition probability densities corresponding to
the path measure $\mathbb{P}^{x;\omega_n}$. In addition, for any $\tilde f\in C(\bbT^d_{M_n})$ we have
 \begin{equation}
 \label{eq.period1}
 \tilde {\bbE}^{x;n}\tilde f(\tilde X_t)=\bbE^{x;\omega_n} f(X_t),\quad x\in \bbT^d_{M_n},
 \end{equation}
 where $ f\in C_b(\bbR^d)$ is the $M_n$-periodic extension of
 $\tilde f$, $\tilde X_t=\pi_{n}(X_t)$ and $\tilde {\bbE}^{x;n}$ is the expectation corresponding
 to $\tilde {\bbP}^{x;n}$. Moreover, the path measure $\tilde {\bbP}^{x;n}$
 {can be characterized as the unique solution to the martingale problem
associated with the following integro-differential} operator
\[
\tilde L^{n} u(x)=\int_{\bbT^d_{M_n}}(u(x+z)-u(x)-\nabla
u(x)\cdot z\mathds{1}_{\{|z|\le 1\}})\,\tilde{\rm n}_n(x,dz), \qquad u\in C^2(\bbT^d_{M_n}),
\]
with a L\'evy kernel $\tilde {\rm n}_n(x,dz)$ of the form
\[
   \tilde {\rm n}_n(x,dz):=\sum_{m\in\bbZ^d}\frac{\langle  {\bf a}(\tau_x\omega_n)(z+M_nm),(z+M_nm)\rangle}{|z+M_nm|^{d+2+\alpha}}\,dz, \qquad x\in\bbR^d.\]

\begin{proposition}
\label{thm:Inv_density_torus}
\begin{enumerate}
\item[(i)]For each $n\in\bbN$, there exists an invariant density $\phi_n$
for the process $ \left(\tilde{\mathbb{P}}^{x;n}\right)_{x\in\bbT^d_{M_n}}$,
i.e.\ $\phi_n\ge 0$, $\ell_{M_n}$-a.e.\ in $\bbT^d_{M_n}$, $\|\phi_n\|_{L^1(\bbT^d_{M_n})}=1$ and
\begin{equation}
\label{082905-20}
\int_{\bbT^d_{M_n}}\left[\tilde{\mathbb{E}}^{x;n}f(\tilde X_t)\right]
\phi_n(x)\, \ell_{M_n}(dx)= \int_{\bbT^d_{M_n}}f(x)\phi_n(x)\, \ell_{M_n}(dx),\qquad t\ge0, \, f\in B_b(\bbT^d_{M_n}).
\end{equation}
\item[(ii)] Let $(\nu_n)_{n\in\bbN}$ be a sequence of  probability measures on
$(\Omega_{\cal r},{\cal B}(\Omega_{\cal r}))$ defined as
\begin{equation}
  \label{nun}
  \nu_n(A)\, := \int_{\bbT^d_{M_n}}\mathds{1}_{A}(\tau_x\om_n)\phi_n(x)\,
  \ell_{M_n}(dx)
  ,\qquad A\in {\cal B}(\Omega_{\cal r}),
  \end{equation}
with $\phi_n$ as in (i). Then, there exists $c_*>0$, depending only on
  $d,\lambda,\Lambda,\alpha$, such that 
\begin{equation}
\label{eq.tightm}
\int_{\Omega_{\cal r}}F\,d\nu_n\le c_*\|F\|_\infty^{1-\alpha/(2d)}\left
  (\int_{\Omega_{\cal r}} F d\mu_n \right)^{\alpha/(2d)}, \qquad F\in B_b(\Omega_{\cal r}).
\end{equation}
\item[(iii)] 
The sequence $(\nu_n)_{n\in \bbN}$ is tight. Any weak limiting point $\nu_*$
of $(\nu_n)_{n\in\bbN}$ satisfies 
\begin{equation}
\label{eq.tightm1}
\int_{\Omega_{\cal r}}F\,d\nu_*\le c_*\|F\|_\infty^{1-\alpha/(2d)}\left
  (\int_{\Omega_{\cal r}} Fd\mu \right)^{\alpha/(2d)},
\end{equation}
for any non-negative $F\in B_b(\Omega_{\cal r})$,
with $c_*$ as in \eqref{eq.tightm}. In consequence, $\nu_*$
is absolutely continuous with respect to $\mu$ and its density
$\Phi_*=\frac{d \nu_*}{d\mu}$ belongs to any $L^p(\mu)$, where $1\le p<\frac{d}{d-\al/2}$.
\end{enumerate}
\end{proposition}
\begin{proof}
Let $(\om_n)_{n\in\bbN}$, $(M_n)_{n\in\bbN}$  be the sequences  
appearing in Theorem \ref{thm:parta}. 
We define a process
$(\hat{\bbP}^{x;n})_{x\in \bbT^d}$ by letting
$
\hat{\bbP}^{x;n}:= ({\cal T}_{M_n^{-1}})_\sharp \tilde{\mathbb
  P}^{xM_n;n},$ 
where ${\cal T}_\varepsilon$ appeared in \eqref{011903-24}. Such a process corresponds to rescaling of the
process $(\tilde{\bbP}^{x;n})$
on $\bbT^d_{M_n}$, i.e.\ 
$\hat{\bbP}^{x;n}=\pi_\sharp \bbP^{x;\omega_n}_{M_n^{-1}}$, where the
family $(\bbP^{x;\omega}_\varepsilon)$ has been introduced in
\eqref{beq01eps} and $\pi$ is the canonical projection of $\bbR^d$
onto the unit torus $\bbT^d$.

{Let $\hat X_t:=\pi(X_t)$} and let $\hat P_{t}^{(n)} f(x):=\hat
\bbE^{x;n}  f\left(\hat X_t\right)$, $t\ge0$, $x\in \bbT^d$ be the
transition probability semigroup on $B_b(\bbT^d)$ corresponding to
$\left(\hat \bbP^{x;n}\right)_{x\in \bbT^d}$. It is $C_b$-Feller,
thanks to the fact that $(\mathbb{P}^{x;\omega_n})$ is Feller and
\cite[Theorem 1.9]{BSW}. Its respective $1$-resolvent operator is
given by 
\begin{equation}
\label{res-1}
\hat R^{(n)}_{1}f(x)=\int_0^{\infty}e^{-t}\hat P_{t}^{(n)}f(x) dt,\quad x\in\bbT^d,\,f\in B_b(\bbT^d)
\end{equation}

{\bf Proof of (i).} Since $\bbT^d$ is compact and
$(\hat{\mathbb{P}}^{x;n})_{x\in\bbT^d}$ is  $C_b$-Feller,  the existence of an invariant probability measure $\hat m_n$ for
the process $(\hat{\mathbb{P}}^{x;n})_{x\in\bbT^d}$ follows from the
Krylov-Bogoliubov theorem (cf.\ \cite[Theorem 3.1.1]{DZ}).  {By Theorem \ref{ABP}, we know that} $\hat m_n$ is
absolutely continuous with respect to $\ell$, the normalized Lebesgue
measure on $\bbT^d$,  {with a density $\hat\phi_n:=d\hat m_n/d\ell$ such that}
$\|\hat{\phi}_n\|_{L^1(\bbT^d)}=1$. We then write
\begin{equation}
\label{082905-20hat}
\int_{\bbT^d}\hat{\mathbb{E}}^{x;n}\left[ f(\hat X_t)\right]
\hat\phi_n(x)\, \ell(dx)= \int_{\bbT^d}f(x)\hat \phi_n(x)\, \ell(dx),\qquad t\ge0, \, f\in B_b(\bbT^d).
\end{equation}
{Let us denote
$\phi_n:= \hat\phi_n\circ j_{M_n}$, where $j_{M_n}(x):= x/M_n$. It is
a density with respect to the normalized Lebesgue measure on $\bbT^d_{M_n}$} 
We claim that  $\phi_n(x)$ is   invariant   for the process
$(\tilde{\mathbb{P}}^{x;n})_{x\in\bbT^d_{M_n}}$.  Indeed, for any
$f\in  B_b(\bbT^d_{M_n})$, we have  by \eqref{082905-20hat}
\begin{multline*}
\int_{\bbT^d_{M_n}} f(x)\phi_n(x) \,\ell_{M_n}(dx)=
   \int_{\bbT^d}f\circ j_{M_n}^{-1}(x)\hat\phi_n(x) \,\ell(dx)
\\
=\int_{\bbT^d}\hat{\bbE}^{x;n} \Big[f\circ
j_{M_n}^{-1}(\hat X_{ tM_{n}^{-\al}})\Big]\hat\phi_n(x) \,\ell(dx)=\int_{\bbT^d_{M_n}}\tilde {\bbE}^{x;n} \Big[f(\tilde X_{ t})\Big]\phi_n(x)\,\ell_{M_n}(dx),
\end{multline*}
and thus, we have concluded the proof of (i).

{\bf Proof of (ii).} Let $F\in B_b(\Omega_{\cal r})$  be non-negative {and such that} $\|F\|_\infty\le 1$. Then,
\begin{equation}
  \label{010507-24}
\int_{\Omega_{\cal r}}F\,d\nu_n= \int_{\bbT^d_{M_n}}  F(\tau_x\om_n)\phi_n(x)\, \ell_{M_n}(dx)
= \int_{\bbT^d} F(\tau_{M_nx}\om_n) \hat \phi_n(x)\, \ell(dx).
\end{equation}
Since  $\hat \phi_n(x)$ is invariant under the measure $(\hat{\mathbb
  P}^{x;n})_{x\in \bbT^d}$, we have $\big(\hat R^{(n)}_{1}\big)^*\hat \phi_n=\hat
\phi_n$. Using this and Theorem \ref{ABP},  the utmost right-hand side
of \eqref{010507-24} can be rewritten as
\[
\begin{split}
\int_{\bbT^d}F(\tau_{M_nx}\om_n)  \big(\hat R^{(n)}_{1}\big)^*\hat
\phi_n(x)\, \ell(dx) 
&=\int_{\bbT^d}\hat R^{(n)}_{1}(F(\tau_{M_n \cdot }\om_n))(x)\hat \phi_n(x)\, \ell(dx)\\
&
\le  \| \hat R^{(n)}_{1}(F(\tau_{M_n\cdot}\om_n))\|_\infty
\le  \bar C\left (\int_{\bbT^d}F(\tau_{M_nx}\om_n) \ell(dx)\right)^{\alpha/(2d)}\\&
=\bar C\left (\int_{\bbT^d_{M_n}}F(\tau_{x}\om_n)
  \ell_{M_n}(dx)\right)^{\alpha/(2d)}=\bar C\left (\int_{\Omega_{\cal r}} F d\mu_n \right)^{\alpha/(2d)},
\end{split}\]
 {where $\bar C$ is the same constant as in  Theorem \ref{ABP}.} This ends the proof of  \eqref{eq.tightm} when $0\le F\le 1$. The estimate follows  for a general non-negative $F\in B_b(\Omega_{\cal r})$ by
considering $F/\|F\|_\infty$. 

{\bf Proof of (iii). } According to Theorem \ref{thm:parta}, the sequence $(\mu_n)_{n\in \bbN}$ weakly converges to 
$\mu$. It is therefore tight: for any $\eps>0$, there exists a compact
set $K\subset \Omega_{\cal r}$ 
such that $\mu_n(K^c)<\eps$ for any $n\in\bbN$. Hence,  by the already
proved estimate \eqref{eq.tightm}, we show that
\[\nu_n(K^c)\le c_* \mu_n^{\al/(2d)}(K^c)\le  c_*\eps^{\al/(2d)},\]
which proves   tightness of the sequence $(\nu_n)_{n\in \bbN}$. Letting $n\to+\infty$ in \eqref{eq.tightm},  we conclude
\eqref{eq.tightm1} for  $F\in C_b(\Omega_{\cal r})$. Since $\Omega_{\cal r}$ is  Polish, by
the Ulam theorem, the measure $\mu+\nu_\ast$ is Radon and therefore $C_b(\Omega_{\cal r})$ is dense
in $L^1(\mu+\nu_*)$, see
e.g. \cite[Proposition 7.9]{Fol}.  Suppose that $F\in B_b(\Omega_{\cal r})$ and $(F_n)\subset
C_b(\Omega_{\cal r})$ is such that
$\lim_{n\to+\infty}\|F_n-F\|_{L^1(\mu+\nu_*)}=0$. By considering $0\vee
(F_n\wedge \|F\|_\infty)$, we can always assume that $0\le F_n\le
\|F\|_\infty$. Each $F_n$ satisfies \eqref{eq.tightm1}. Letting
$n\to+\infty$, we conclude the estimate for the limiting $F$.
From \eqref{eq.tightm1}, we infer that $\nu_\ast$ is absolutely
continuous with respect to $\mu$. According to the estimate, its density $\Phi_\ast$ satisfies
\begin{equation*}
\la \mu[ \Phi_\ast \ge \la]\le  \int_{\Omega_{\cal r}}\mathbf 1_{[ \Phi_\ast \ge \la]}\Phi_\ast d\mu\le c_* 
 \mu[ \Phi_\ast \ge \la]^{\alpha/(2d)},\quad \la\ge0.
\end{equation*}
Therefore,
\begin{equation*}
\mu[ \Phi_\ast \ge \la]\le  \frac{C}{ \la^{2d/(2d-\alpha)}},
\end{equation*}
for some constant $C>0$. Thus,
\begin{equation}
  \label{eq.tightm2}
  \int_{\Om} \Phi_\ast^p d\mu=p\int_0^{+\infty}\la^{p-1}\mu[ \Phi_\ast \ge
  \la]d\la\le p+C\int_1^{+\infty}\frac{  d\la}{ \la^{1+2d/(2d-\alpha)-p}}<+\infty,
\end{equation}
provided that $1\le p<2d/(2d-\alpha)$,  which ends the proof of part (iii).
\end{proof}

In order to conclude the proof of Theorem \ref{thm:main}, we need the following lemma 
asserting the $C_b$-Feller property for the semigroup generated by the environment process.
\begin{lemma}
  \label{feller}
  The semigroup $({\frak P}_{t})_{t\ge0}$ given by \eqref{Pt}
  is $C_b$-Feller, i.e.
\[
{\frak P}_{t}\left(C_b(\Omega_{\cal r})\right)\subseteq C_b(\Omega_{\cal r}), \qquad t>0.\]
\end{lemma}
\begin{proof}
Fix $t>0$ and  $F\in C_b(\Omega_{\cal r})$. Let $(\om_n)_{n\in\bbN}\subset \Omega_{\cal r}$ and $\om\in \Omega_{\cal r}$ be such that ${\rm d}_{\cal r}(\om_n,\om)\to0$, as $n\to\infty$.
Then the respective path measures $\bbP^{0,\om_n}$ converge to  $\bbP^{0,\om}$, weakly over ${\cal D}$.
Tightness of the laws of $(X_t)$ under $\bbP^{0;\om_n}$   implies in
particular that for any $\eps>0$ 
there exists a compact set $K\subset\bbR^d$ such that
\[\bbP^{0;\om_n}\left(X_t\not\in K\right)<\eps,\,n\ge1\quad\mbox{and}\quad \bbP^{0,\om}\left(X_t\not\in K\right)<\eps.\]
Then,
\[| {\frak P}_{t} F(\om_n)- {\frak P}_{t} F(\om)|\le \big|\bbE^{0;\om_n}[F(\tau_{X_t}\om_n),\, X_t\in K]
-\bbE^{0;\om}[F(\tau_{X_t}\om),\,X_t\in K]\big|+2\eps\|F\|_\infty.\]
The function $F$, when restricted to 
the compact set ${\cal K}(K):=\{\tau_y(\om)\colon y\in K,\,\om\in {\cal K}\}$, where
${\cal K}:=\{\om,\om_1,\om_2,\ldots\}$, is uniformly
continuous. Since ${\rm d}_{\cal r}(\om_n,\om)\to0$, we have
\begin{equation}
  \label{041910-23}
\lim_{n\to\infty}\big|\bbE^{0;\om_n}[F(\tau_{X_t}\om_n),\,X_t\in K]-\bbE^{0;\om_n}[F(\tau_{X_t}\om),\,X_t\in K]\big|=0.
\end{equation}

The function $y\mapsto F(\tau_y\om)$ is bounded and continuous on
$\bbR^d$ for each fixed $\om$. Since $(\bbP^{0;\om})$ is a Feller, c\`adl\`ag process, its paths are quasi-left-continuous, see e.g.\ \cite[Proposition 1.2.7, p. 21]{bertoin}. 
Thus,
 $\bbP^{0;\om}[F(\tau_{X_t}\om)=F(\tau_{X_{t-}}\om)]=1$. From
  the above and the fact that
$\bbP^{0;\om_n}$ weakly converge to $\bbP^{0;\om}$ in $\mathcal{D}$, 
\cite[Theorem 2.7]{bil} implies that
\begin{equation}
  \label{051910-23}
\lim_{n\to\infty}\big[\bbE^{0;\om_n}[F(\tau_{X_t}\om)]-\bbE^{0;\om}[F(\tau_{X_t}\om)]\big]=0.
\end{equation}



 


Hence,
\begin{equation}
  \label{061910-23}
\limsup_{n\to\infty}\big|\bbE^{0;\om_n}[F(\tau_{X_t}\om),\,X_t\in
K]-\bbE^{0;\om}[F(\tau_{X_t}\om),\,X_t\in K]\big|\le 2\eps \|F\|_\infty.
\end{equation}
Summarizing, we have shown
\[
\limsup_{n\to\infty}|{\frak P}_{t}F(\om_n)-{\frak P}_{t}F(\om)|\le
4\eps\|F\|_\infty,\quad  \eps>0. \]
Thus, $|{\frak P}_{t}F(\om_n)-{\frak P}_{t}F(\om)|\to 0
$, as $n\to\infty$,
and the conclusion of the lemma follows. 
\end{proof}


\subsection{Conclusion of the proof of Theorem \ref{thm:main}}

\label{sec4.5}

We show that  any
weak limiting point $\nu_*$ for the sequence $(\nu_n)_{n\in\bbN}$, defined in \eqref{nun},
is invariant for the  {environment} process  $(\eta_t)_{t\ge 0}$  introduced in \eqref{eta-t}. Let us fix $F$ in $C_b(\Omega_{\cal r})$. We have
\[
 \int_{\Omega_{\cal r}}{\frak P}_{t} F(\om)\, \nu_n(d\om)= \int_{\Om}\bbE^{0;\om} F(\eta_t(\om))\, \nu_n(d\om)=
  \int_{\bbT^d_{M_n}}  \bbE^{0;\tau_x\om_n} F\left(\tau_{X_t}\tau_x\om_n\right)\phi_n(x)\, \ell_{M_n}(dx) , \qquad n \in
  \bbN.
\]
By virtue of \eqref{022905-20} and \eqref{eq.period1} we can 
rewrite the utmost right-hand side  as
\begin{align*} \int_{\bbT^d_{M_n}}  \bbE^{x; \om_n}
  F\left(\tau_{X_t}\om_n\right)\phi_n(x)\, \ell_{M_n}(dx) = \,
  \int_{\bbT^d_{M_n}}\tilde \bbE^{x;n} F\left(\tau_{\tilde X_t}
    \om_n\right)\phi_n(x)\, \ell_{M_n}(dx).
\end{align*}
Using
\eqref{082905-20},
we conclude that the right hand side equals  
 \begin{align*}
  \, \int_{\bbT^d_{M_n}}F\left(\tau_x\om_n\right)\phi_n(x)\,
  \ell_{M_n}(dx)= \int_{\Omega_{\cal r}}F\left(\om\right)\nu_n(d\om).
 \end{align*}
Since $ {\frak
  P}_{t}F\in C_b(\Omega_{\cal r})$, see Lemma \ref{feller}, from the weak convergence of
$(\nu_n)$ to $\nu_*$ and the above argument  
\[
 \int_{\Omega_{\cal r}} {\frak
  P}_{t}F\, d\nu_*= \lim_{n\to\infty}\int_{\Omega_{\cal r}}{\frak
  P}_{t}F\, d\nu_n\,
=\, \lim_{n\to\infty}\int_{\Omega_{\cal r}}F \, d\nu_n\,=\,
\int_{\Omega_{\cal r}}Fd\nu_* ,
\]
which proves the invariance
of any weak limiting point $\nu_*$.

Thanks to part  (iii) of  Proposition \ref{thm:Inv_density_torus}, $\nu_*$ -  any  weak limiting point
$(\nu_n)_{n\ge1}$ - is absolutely continuous w.r.t.\ $\mu$.
We
claim that $\Phi_*:=\frac{d\nu_*}{d\mu}>0$, $\mu$-a.s.\ in $\Omega_{\cal r}$. 
Indeed, let $A:=\{\om\in
\Omega_{\cal r}\colon \Phi_*(\om)= 0\}$. Clearly, $\mu(A)<1$, as $\Phi_*$ is a
density w.r.t.\ $\mu$. Suppose that
$\mu(A)>0$.  Let
\begin{equation}
  \label{022409-24}
  \phi_A(y;\om)=\mathds{1}_A(\tau_{y}\om)\quad\mbox{and}\quad
  B^\om:=\{y\in \bbR^d\colon\tau_{y}\om\in A\}.
\end{equation}
We have, see \eqref{032310-23}-\eqref{012812-23},
\begin{multline*}
0=\int_0^\infty e^{-t} \,dt\left(\int_{\Omega_{\cal r}}\mathds{1}_A\Phi_*\,  d\mu\right) 
=   \int_0^\infty e^{-t} \,dt\left(\int_{\Omega_{\cal r}}{\frak P}_{t} \mathds{1}_A\Phi_*\,  d\mu\right)\\
= \int_{\Omega_{\cal r} } \Phi_*(\om)\mu(d\om) \int_0^\infty e^{-t} P_{t}^{\om}\phi_A (0;\om)dt=\int_{A^c }  
R_{1}^{\om}\phi_A (0;\om) \Phi_*(\om)\mu(d\om).
\end{multline*}
Thus,
\begin{equation}
  \label{012409-24}
R_{1}^{\om}\phi_A (0;\om)=0\quad\mbox{ for $\mu$-a.e.
  $\om\in A^c$.}
\end{equation}
Let $(x_n)$ be a dense subset of $\bbR^d$ and $\nu$ be Borel
probability measure on $\bbR^d$ given by 
$\nu(dy):=\sum_{n=1}^\infty2^{-n}\delta_{x_n}(dy).$
We   let ${\frak m}:=\nu R_{1}$, i.e.
\begin{equation}
  \label{rm}
\int_{\bbR^d}f\,d{\frak m}= \int_{\bbR^d}R_1f\,d\nu,\qquad f\in B_b(\bbR^d).
 \end{equation}
Thanks to Proposition \ref{prop012812-23} condition \eqref{012409-24} implies
that for $\mu$-a.e.\ $\om\in A^c$, ${\frak m}(B^\om)=0$, where $B^\om$
is defined in \eqref{022409-24}. In other words
$$
\int_{\bbR^d}\mu\left(\tau_{-y}(A)\setminus A\right) {\frak m}(dy)=0.
$$
Therefore, there
  exists a Borel subset $Z$ of $\bbR^d$ such that $\frak{m}(Z)=1$ and
  $\mu\left(\tau_{-y}(A)\setminus A\right)=0$
for all $y\in Z$. According to  Proposition \ref{rem.appb}, the set $Z$ is dense in
$\bbR^d$. Thanks to \eqref{eq.inv} and the continuity of  
$x\mapsto 1_A\circ\tau_x$ in $L^1(\mu)$, it  follows that
$\mu\left(\tau_y(A)\Delta A\right)=0$ for all $y\in \bbR^d$, where
$\Delta$ denotes the symmetric difference of sets.
This in turn
implies that $\mu(A)\in\{0,1\}$, due to ergodicity of $\mu$ under $\tau_x$,
which contradicts the assumption that $0<\mu(A)<1$.
Thus, we have shown that  $\Phi_*>0$, $\mu$-a.s.\ in $\Omega_{\mathfrak{r}}$. Actually, the above argument proves that the density of any invariant measure, which is absolutely continuous with respect to $\mu$, has to
be strictly positive, $\mu$ a.s.\ in in $\Omega_{\mathfrak{r}}$.
Ergodicity of $\nu_*$ is then an easy consequence of this fact. Indeed, if there had existed $A$ such that $\nu_*(A)\in (0,1)$  and
$\mathds{1}_A\left(\eta_t(\om)\right) =\mathds{1}_A(\om)$,
$\nu_*$-a.s.\ in $\Omega_{\cal r}$, then  both
measures $\nu_*(A)^{-1}\mathds{1}_A\Phi_*d\nu_*$ 
and $\nu_*(A^c)^{-1}\mathds{1}_{A^c}\Phi_*d\nu_*$
would have been invariant and of disjoint supports, which leads
to a contradiction. This ends the proof of Theorem \ref{thm:main}.\qed

\subsection{ Proof of Theorem \ref{thm:main} in the general case }
\label{sec.main.gen}

In the present section we shall dispense with the additional regularity assumptions on the coefficients
that has been  made in the previous section.
Inspecting the proof of Lemma  \ref{feller} we can see that the
argument    required only that  $(\om_n)_{n\in\bbN}\subset \Omega^{\cal M}$ and $\om\in \Omega^{\cal M}$.
Therefore, we can conclude the following variant of the lemma.
\begin{lemma}
  \label{feller1}
Suppose that $F\in C_b(\Om^{\mathcal M})$ and $t\ge0$. Then ${\frak P}_{t}F _{|\Om^{\mathcal
    M}}\in C_b(\Om^{\mathcal M})$.
\end{lemma}

Let 
\[
I_{\cal r}(\omega)(x):= {\cal r}{\rm I_d}+(j_{{\cal r}}\ast \omega)(x),\quad x\in\bbR^d,\, \omega\in \Omega,
\]
where $j_{\cal r}$ is a standard smooth mollifier, i.e. $j_{\cal
  r}(x)={\cal r}^{-d} j(x/{\cal r})$, with $j\in
C_c^\infty(\bbR^d)$, non-negative and $\int_{\bbR^d}j(x)dx=1$.
We have $I_{\cal r}:\Omega\to \Omega_{{\cal r}}$. 
Let $\omega_{\cal r}:= I_{\cal r}(\omega)$ and   $\mu^{\cal
  r}=(I_{\cal r})_{\sharp}\mu$.
We see that $\sup_{\omega\in \Omega}{\rm d}\big(I_{\cal
  r}(\omega),\omega\big)\to0$ if $\cal r\downarrow 0$. 
  Recall that by \cite[Theorem 8.3.7]{cohn}, $\Omega_{{\cal r}}$ is a Borel measurable subset of $\Omega$. 
Each measure   $\mu^{\cal r}$ can be extended to $\mathcal B(\Omega)$ by letting
$\mu^{\cal r}\big(\Omega\setminus \Omega_{\cal r}\big)=0$. 
Moreover,  the space $\Omega^{\cal M}$, metrized with the
metric obtained by the restriction of ${\rm d}$, is 
separable.

By Theorem \ref{thm:main}, for each $\cal r>0$, there exists an invariant,  ergodic measure
$\mu_*^{\cal r}$ on $\Omega_{\cal r}$ for the semigroup $(\mathfrak
P_{t})_{t\ge 0}$ on $\Omega_{\cal r}$. The measure is absolutely continuous
w.r.t.\  $\mu^{\cal r}$ and 
$d\mu_*^{\cal r} =\Phi_*^{\cal r}\,d\mu^{\cal r}$, where $\Phi_*^{\cal
  r}$ is positive $\mu^{\cal r}$-a.s.\ in $\Omega_{\cal r}$.
We extend $\mu_*^{\cal r}$ to $ {\cal B}(\Omega)$ by letting it equal to
zero on $\Omega\setminus \Omega_{\cal r}$. Observe that  $\mu^{\cal r}$ converge weakly to $\mu$, as ${\cal r}\downarrow0$. 
By  \eqref{eq.tightm1}, we have for any non-negative $F\in B_b(\Omega)$,
\[
\begin{split}
\int_{\Omega}F\,d\mu^{\cal r}_*&= \int_{\Omega_{\cal r}}F\,d\mu^{\cal r}_*
\le c\|F\|_\infty^{1-\frac{\alpha}{2d}}\left [\int_{\Omega_{\cal r}} F(\omega)\mu^{\cal r}(d\omega)
\right]^{\frac{\alpha}{2d}}
=c\|F\|_\infty^{1-\frac{\alpha}{2d}}\left(\int_{\Omega} F(I_{\cal r}(\omega))\mu(d\omega)
\right)^{\frac{\alpha}{2d}},
\end{split}
\]  
with $c$ independent of ${\cal r}\in (0,1]$.
Thus, up to a sub-sequence, $\mu^{\cal r}_*$ converges weakly, as ${\cal r}\to0$,  over ${\Om}$ to some measure $\mu_*$ which satisfies
\[
 \int_{\Omega}F\,d\mu _* \le c\|F\|_\infty^{1-\frac{\alpha}{2d}}\left
   [\int_{\Omega} Fd\mu 
\right]^{\frac{\alpha}{2d}},\]
for all non-negative $F\in B_b(\Omega)$. This, in turn, implies that $\mu_*$ is absolutely continuous with respect to $\mu$ and $\Phi_*:= d\mu_*/d\mu$ belongs to $L^p(\mu)$ for any $p\in [1,d/(d-\al/2))$ (cf.\ the proof of Proposition \ref{thm:Inv_density_torus}(iii)). Since ${ \Om}^{\cal M}$ is dense in ${  \Om}$, we infer that measures $\mu^{\cal r}_*$, restricted to   ${ \Om}^{\cal M}$,  converge weakly, as ${\cal r}\to0$, to the restriction of $\mu_*$. Indeed,  any $F\in C_b( {\Om}^{\cal M})$ that is uniformly continuous can be uniquely extended to $\tilde F\in C_b({\Om})$. Therefore (with some abuse of notation, we denote by the same symbol for measures and their restrictions) 
$$
\lim_{{\cal r}\downarrow0}\int_{{  \Om}^{\cal M}}Fd \mu^{\cal
  r}_*=\lim_{{\cal r}\to0}\int_{{  \Om}}\tilde Fd \mu^{\cal r}_*=
\int_{{  \Om}}\tilde Fd \mu_*=\int_{{  \Om}^{\cal M}}Fd \mu_*.
$$
By virtue of \cite[Theorem 2.1]{bil} the above implies the  
convergence in question.
Consequently, by Lemma \ref{feller1}
we conclude that
\[ \int_{\Omega}\mathfrak P_{t} F\,d\mu_*=\int_{\Omega^{\cal M}}\mathfrak P_t F\, d\mu_*= \lim_{{\cal r}\downarrow0}\int_{\Omega^{\cal M}}\mathfrak P_{t} F\,d\mu_*^{\cal r} =\lim_{{\cal r}\downarrow0}\int_{\Omega_{\cal r}}\mathfrak P_{t} F\,d\mu_*^{\cal r}.\]
Using the version of Theorem \ref{thm:main}  already proved on $\Om_{\cal r}$ we infer that the utmost right hand side equals
\[ \lim_{{\cal r}\downarrow0}\int_{\Omega_{\cal r}} F\,d\mu_*^{\cal r}=\lim_{{\cal r}\downarrow0}\int_{\Omega } F\,d\mu_*^{\cal r}=\int_{\Omega} F\,d\mu_*. \]

The proof that $\mu_*$ is ergodic can be conducted in the same way as 
in  Section  \ref{sec4.5}.

\appendix

\section{ Aleksandrov-Bakelman-Pucci estimates}
\label{sec-abp}

Let $(\bbP^x)_{x\in\bbR^d}$ be a Feller process corresponding to the generator $L$ of the form \eqref{Lom}, where
the matrix valued function $x\mapsto a(x)=[a_{i,j}(x)]_{i,j=1}^d$ satisfies the assumptions
made in Section \ref{sec1}. In the present section we do not assume
that the coefficients $a(x)$ are random.

Let $D\subset \bbR^d$ be  an open set.  Define
the exit time $ \tau_D:{ {\mathcal D}}\to[0,\infty]$ of the canonical
process $( X_t)_{t\ge0}$, see \eqref{eq:can},  from $D$ as
\[
\tau_D:=\inf \left\{t>0\colon X_t\not\in D \right\}.
\]
It is a stopping time,
i.e.\ $\{\tau_D\le t\}\in {\mathcal F}_t$ for any
$t\ge0$. The transition semigroup on $D$ with
the null exterior condition is defined  as
\[
P^D_tf(x):=\bbE^x\left[f( X_t), \, t<\tau_D\right],\quad t\ge 0,\,
x\in D,\,  f\in B_b(D).
\]
Furthermore, for  any $\beta>0$, we introduce the  $\beta$-resolvent of
$L$ on $D$ for any
  $f\in B_b(D)$   by letting
\[
R^D_\beta  f(x):=\int_0^\infty e^{- \beta t}P^D_tf(x)\,dt
=\bbE^x\left[\int_0^{\tau_D}e^{-  \beta t}f( X_t)\,dt\right],\quad x\in D.
\]
If we suppose that
\begin{equation}
  \label{eq:finet}
  \mathbb{E}^x\tau_D<\infty ,\quad x\in D,
  \end{equation}
  then, we can define the resolvent $R^D  f $ also for $\beta=0$.
Without the assumption \eqref{eq:finet} the resolvent $R^D  f(x)$ can be
defined for a non-negative $f$ (not necessarily bounded), but we have
to admit the possibility that it equals $\infty$ for some $x$.

Let
$\pi: \bbR^d\to \bbT^d$
be the canonical projection of $\bbR^d$ onto $\bbT^d$.
Let  $C_{\rm per}(\bbR^d)$ be the space of all continuous
functions which are $1$-periodic in each variable.
There is a one-to-one correspondence between $C_{\rm  per}(\bbR^d)$ and $C(\bbT^d)$, i.e. for every $\tilde f\in
C(\bbT^d)$,
there exists a unique $f\in
C_{\rm
  per}(\bbR^d)$ such that $ f(x)=\tilde f\circ\pi(x)$,
$x\in\bbR^d$.

Let us suppose that the function $x\mapsto a(x)$ is
$1$-periodic, as it is the case discussed in Section
\ref{sec4.4}. Thus, the corresponding transition probability semigroup
$(P_t)_{t\ge0}$ has the following property
\[P_t \Big(C_{\rm per}(\bbR^d)
\Big)\subseteq C_{\rm per}(\bbR^d), \qquad t>0.\]
The  semigroup $(P_t)_{t\ge0}$ induces then a   strongly continuous semigroup $(\tilde P_t)_{t\ge0}$  on
$C(\bbT^d)$, by virtue of \cite[Lemma 1.18]{BSW}.
The respective process $\tilde{\mathbb{P}}^{x}: =(\pi)_\sharp
\mathbb{P}^{x}$, $x\in\bbT^d$, is strongly Markovian and Feller  on
 $\bbT^d$, with transition probabilities given by
 \[
 \tilde P_{t}(x,dy)= \sum_{m\in\bbZ^d}  P_{t}(x,dy+m ),\quad t>0,\, x,y\in \bbT^d.
\]
Here, $P_{t}(x,dy)$ are the transition probabilities corresponding to
the path measure $\mathbb{P}^{x}$.

 In addition, for any $\tilde f\in C(\bbT^d)$, we have
\[\tilde {\bbE}^{x}\tilde f(\tilde X_t)=\bbE^{x} f(X_t),\qquad t\ge0,\, x\in \bbT^d,\]
 where $ f\in C_{\rm per}(\bbR^d)$ is the $1$-periodic extension of
 $\tilde f$, $\tilde X_t=\pi(X_t)$ and $\tilde {\bbE}^{x}$ is the expectation corresponding
 to $\tilde {\bbP}^{x}$.
The $1$-resolvent corresponding to $(\tilde P_t)_{t\ge0}$ is then
given by
\[\tilde R_1\tilde f(x): =\int_0^{\infty}e^{-t}\tilde P_t\tilde
  f(x)dt,\qquad x\in\bbT^d,\, \tilde f\in B_b(\bbT^d).\]

\begin{thm}
\label{ABP}
We have

\[
  \Vert  \tilde  R_1\tilde f \Vert_{L^\infty(\bbT^d)} \le \bar{ C} \Vert
 \tilde  f\Vert_{L^\infty(\bbT^d)}^{(1-\alpha/2)}\Vert
 \tilde  f\Vert_{L^d(\bbT^d)}^{\alpha/2},\qquad \tilde f\in L^p(\bbT^d),\]
where  the constant $\bar C$ depends only on $\alpha,d,\lambda,\Lambda$.
\end{thm}
\proof

Recall that $B(x,r)$ denotes a ball centered at $x$ with radius $r>0$.
Let $D:= \{x\in\bbR^d\colon {\rm dist}(x,Q_1)<1\}$.  Suppose also that
$f\in C_{\rm per}(\bbR^d)$ is such that $ f=\tilde f\circ\pi$.
Using the strong Markov property, we can then write for any $x\in\bbT^d$:
\[
   \tilde  R_1\tilde f  (x)= \mathbb
   E^{x}\left[\int_0^{\tau_{D}}e^{-t}  f( X_t) \, dt\right]
  +\bbE^{x }\left[e^{-\tau_{D}}\int_0^{\infty}e^{-t}\mathbb
    E^{X_{\tau_{D}} }\left[  f( X_t) \right]\, dt\right].
\]
We  have then
\begin{equation}
\label{010807-24}
\sup_{x\in\bbT^d}\vert \tilde  R_1\tilde f  (x)\vert \le
\sup_{x\in\bbT^d} \bbE^{x } \left[\int_0^{\tau_{D}}e^{-s}|  f( X(s))| \, ds\right] +\sup_{x\in\bbT^d} \bbE^{x}\left[e^{-\tau_{D}}
\right]\sup_{x\in\bbT^d} \left|\bbE^{x}\int_0^{\infty}e^{-s}  f( X(s)) \, ds\right|.
\end{equation}
Since
$$
u(x)=\bbE^{x } \left[\int_0^{\tau_{D}}e^{-s}|  f( X(s))|
  \, ds\right]
$$ is a non-negative solution of the equation
\[
(u-Lu)(x)=|f(x)|,\quad x\in D,\quad u(x)=0,\quad x\not \in D,
\]
by \cite[Theorem 1.3]{ABP} there exists a constant $C>0$ such
that the first term on
the right hand side of \eqref{010807-24} can be estimated by 
$$
C \Vert
  f\Vert_{L^\infty(D)}^{(1-\alpha/2)}\Vert
 f\Vert_{L^d(D)}^{\alpha/2} .
 $$
 In consequence, 
 \[
   \sup_{x\in\bbT^d}\vert \tilde  R_1\tilde f  (x)\vert\le\,  C \Vert
  f\Vert_{L^\infty(D)}^{(1-\alpha/2)}\Vert
 f\Vert_{L^d(D)}^{\alpha/2}  +
\sup_{x\in\bbT^d}\bbE^{x }\left[e^{-\tau_{D}}\right] \sup_{x\in\bbT^d}\vert \tilde  R_1\tilde f  (x)\vert.
\]
Since $D\subseteq [-3,3]^d$, we can now use  periodicity of $  f$
  to conclude that
\begin{equation}
  \label{040709-23}
  \gamma
  \sup_{x\in\bbT^d}\vert  \tilde  R_1\tilde f  (x)\vert\le
    3^{d}C \Vert
 \tilde  f\Vert_{L^\infty(\bbT^d)}^{(1-\alpha/2)}\Vert
 \tilde  f\Vert_{L^d(\bbT^d)}^{\alpha/2},
\end{equation}
where
\begin{equation}
  \label{gamma}
 \gamma:=  1-\sup_{x\in\bbT^d}\bbE^{x}\left[e^{-\tau_{D}}\right].
  \end{equation}
Since   $B(x,1)\subseteq D$ for any $x\in\bbT^d$, by
\cite[Proposition 3.8]{Kuhn} we infer that  for each $t\in (0,1]$
\begin{align*}
&1-\bbE^{x }\left[e^{-\tau_{D}}\right] \ge 1-\bbE^{x } \left[e^{-\tau_{B(x,1)}}\right] \ge (1-e^{-t})\mathbb{P}^{x}(\tau_{B(x,1)}>t)\\
&\ge \, (1-e^{-t})\left(1-C(d)tq_*\right).
\end{align*}

Here a positive constant $C(d)$ depends only on the dimension
$d$ and
\[
q_*:=\sup_{x\in\bbT^d}\sup_{y\in B(x,1)}\sup_{|\xi|\le
  1}|q(y,\xi)|\le \Lambda(dC_{d,\al}+c_{d,\al}),
\]
see \eqref{beq01eps}

Finally, if one chooses $t:=(2C(d) \Lambda(dC_{d,\al}+c_{d,\al}))^{-1}$, then
$$
\gamma>\frac12\Big(1-e^{-t}\Big)
$$
and the conclusion of the theorem  follows.
\qed

\section{Some property of the resolvent}
\label{appC}

 {As in the previous section, let $(\bbP^x)_{x\in\bbR^d}$ be a Feller process corresponding to the generator $L$ of the form \eqref{Lom}, where
the matrix valued function $x\mapsto a(x)=[a_{i,j}(x)]_{i,j=1}^d$ satisfies the assumptions
made in Section \ref{sec1}.} By \cite[Theorem 1.3]{ABP}, 
$(\bbP_x)_{x\in\bbR^d}$ is in fact a strongly Feller process.
The set $N_x:=\{y\in\bbR^d: \langle a(x)y,y\rangle=0\}$
is a linear space and by \eqref{eq.MH1}  it  cannot be the entire $\bbR^d$.
As a result $\text{\rm supp}\,{\rm n}(x,\cdot)=\bbR^d$ for each $x\in\bbR^d$. 
Recall that measure ${\frak m}$ has been introduced in  \eqref{rm}.

 \begin{proposition}
 \label{rem.appb}
 {Let $B$ be a Borel subset of $\bbR^d$ such that ${\frak
     m}(B)=0$.} Then $B^c$ is a dense subset of $\bbR^d$.
\end{proposition}
\proof
Suppose by contradiction that there exists an open ball
$B(x_0,r_0)\subseteq B$. Then,
\[
R_1\mathds{1}_{B(x_0,r_0)}(x_n)=0,\quad n\ge 1.
\]
By the strong Feller property of $(\bbP_x)_{x\in\bbR^d}$, the function
$x\mapsto R_1\mathds{1}_{B(x_0,r_0)}(x)$ is
continuous and thus, $R_1\mathds{1}_{B(x_0,r_0)}(x)=0,\, x\in\bbR^d$.
As a result, 
\[
0=R_1\mathds{1}_{B(x_0,r_0)}(x_0)=\bbE^{x_0}\int_0^\infty e^{-t} \mathds{1}_{B(x_0,r_0)}(X_t)\,dt
\ge \bbE^{x_0}\int_0^{\tau_{B(x_0,r_0)}} e^{-t} \,dt,
\]
which leads to a contradiction.
\qed

\begin{proposition}
\label{prop012812-23}
{Let $B$ be a Borel subset of $\bbR^d$ such that  ${\frak m}(B)>0$.} Then,
\[R_1\mathds{1}_B(x)>0,\qquad \mbox{for all }x\in\bbR^d.\]
\end{proposition}
\begin{proof}
Let $g\in C_b(\bbR^d)$ be non-negative. By \cite[Lemma 3.5]{CK} (see also \cite[Proposition 1]{barles-imbert}) $u:=R_1g$
is a viscosity solution to 
\begin{equation}
\label{eq.visA}
-Lu+u=g\quad\text{in }\bbR^d.
\end{equation} 
Let $f\in L^d(\bbR^d)$, and let $(f_n)_{n\in\bbN}\subset C_b(\bbR^d)$ be such that $0\le f_n\le 1,\, n\ge 1$ 
and $\|f_n-f\|_{L^d}\le 1/n,\, n\ge 1$. By \cite[Theorem 1.3]{ABP}, we have for any $r\ge 0$
\[
\sup_{|x|<r}|R_1f_n(x)-R_1f(x)|\to 0,\quad n\to \infty.
\]
Consequently, by \cite[Theorem 1]{barles-imbert}  $R_1f$ is a viscosity supersolution to \eqref{eq.visA}
with $g$ replaced by $0$. By \cite[Theorem 2]{Ciomaga} if $R_1f(x)$ for some $x\in\bbR^d$,
then $R_1f\equiv 0$. By \cite[Theorem 4.10]{KlKo} once we prove that
$e^{-t} {\frak m}P_t \le {\frak m},\, t>0$, then 
the asserted implication follows. The last inequality, however, is a direct consequence of the definition of ${\frak m}$.
\end{proof}

\section{Calculation on Fourier symbol}
\label{appC1}

 Using the spherical change of coordinates $y=\ell
\boldsymbol{\sigma}$, $\ell>0$, $\boldsymbol{\sigma}\in \mathbb S^{d-1}$,  in the integral, we can write
\[\mathbf q(\xi;\omega) =|\xi|^{\al}\int_0^{+\infty}\frac{d\ell}{\ell^{\al+1}}\int_{\mathbb
    S^{d-1}} F(\ell \boldsymbol{\sigma} \cdot\hat\xi) \langle
  {\mathbf a}(\omega) \boldsymbol{\sigma} , \boldsymbol{\sigma}
  \rangle S(d \boldsymbol{\sigma}),\qquad
  (\xi,\om)\in \bbR^{d}\times \Omega,\]
where $F(t)=1 -\cos(t),\, t\in\bbR$, and  $\hat\xi=\xi/|\xi|$.
Observe  that
\[
\begin{split}
\int_{\mathbb S^{d-1}} F(\ell \boldsymbol{\sigma} \cdot\hat\xi) \langle {\mathbf a}(\omega) \boldsymbol{\sigma} , \boldsymbol{\sigma}\rangle S(d \boldsymbol{\sigma})&=\sum_{i\neq j}\int_{\mathbb S^{d-1}} F(\ell \boldsymbol{\sigma} \cdot\hat\xi) {\mathbf a}_{ij}(\omega) \boldsymbol{\sigma}_i\boldsymbol{\sigma}_j S(d \boldsymbol{\sigma})\\
&\qquad\qquad\qquad+\sum_{i=1}^d \int_{\mathbb S^{d-1}} F(\ell \boldsymbol{\sigma} \cdot\hat\xi) {\mathbf a}_{ii}(\omega) (\boldsymbol{\sigma}_i^2-\frac1d|\boldsymbol{\sigma}|^2)S(d \boldsymbol{\sigma})\\
&\qquad\qquad\qquad+ \frac1d {\rm Tr}\Big({\mathbf a}(\omega)\Big) \int_{\mathbb S^{d-1}} F(\ell \boldsymbol{\sigma} \cdot\hat\xi)\, S(d \boldsymbol{\sigma})
\end{split}
\]
 {Note that $\boldsymbol{\sigma}_i\boldsymbol{\sigma}_j$ (for $i\not=j$), 
$\boldsymbol{\sigma}_i^2-\frac{1}{d}|\boldsymbol{\sigma}|^2$  (for $i=1,\ldots,d$)  and the constants 
are harmonic polynomials in $\boldsymbol{\sigma}=(\boldsymbol{\sigma}_1,\ldots,\boldsymbol{\sigma}_d)\in\bbR^d$.} By the
Hecke-Funck theorem, see e.g.\ \cite[p. 181]{Hoch}, we can then conclude that
\[\mathbf q(\xi;\omega) =c_d |\xi|^{\al}\langle
  {\mathbf a}(\omega) \hat\xi, \hat\xi\rangle +C_d {\rm Tr}\Big({\mathbf a}(\omega)\Big)|\xi|^{\al},\]
where
\[
\begin{split}
c_d&=|\mathbb S^{d-1}|\int_0^{+\infty}\frac{d\ell}{\ell^{\al+1}}\int_{-1}^1\big(1-\cos(\ell t)\big)P_2(t)(1-t^2)^{\frac{d-3}{2}}dt,\\
C_d&= |\mathbb B^{d}|\int_0^{+\infty}\frac{d\ell}{\ell^{\al+1}}\int_{-1}^1\big(1-\cos(\ell t)\big)\Big(P_0(t)-P_2(t)\Big)(1-t^2)^{\frac{d-3}{2}}dt 
\end{split}
\]
and $P_n(t)=\cos(n \arccos t)$ is the $n$-th degree Tchebyshev polynomial, while $|\mathbb S^{d-1}|$, $|\mathbb B^{d}|$ denote the area of the unit sphere and the volume of the unit ball in $\bbR^d$, respectively.

\section{Borel measurability of  $\Omega^{\mathcal M}$}
\label{appD}

We shall show  that $\Omega^{\cal M}$, introduced in Section
\ref{sec2.2a}, is a Borel measurable subset of $\Omega$. For any
$x\in\bbR^d$, $\omega\in\Omega$, we denote by $\Pi_{x;\omega}$ the set
of all solutions of the martingale problem for $L^\omega$ with the
initial measure $\delta_x$. By \cite[Corollary 3.2 and Lemma
2.1]{kuhn}, the definition of $\Om$ ensures that the set
$\Pi_{x;\omega}$ is non-empty for any
$(x,\om)\in\bbR^d\times\Omega$. For any  $f\in  C_0(\bbR^d)$ and $\lambda>0$ we  introduce
\[
\pi^{\star}_\lambda(x,f;\omega):=\sup_{\bbP\in\Pi_{x;\omega}}\bbE^{\bbP}\int_0^\infty e^{-\lambda t}f(X_t)\,dt,\qquad
\pi_{\star}^\lambda(x,f;\omega):=\inf_{\bbP\in\Pi_{x;\omega}}\bbE^{\bbP}\int_0^\infty e^{-\lambda t}f(X_t)\,dt.
\]
By \cite[Theorem 4.1]{Kuhn}  for any $f\in C_0(\bbR^d)$,
$\omega\in\Omega$ and $\lambda>0$ there exists a strong Markov process
$(\bbQ^{x;\omega}_{f,\lambda}, x\in\bbR^d)$ such that \tk{each
    $\bbQ^{x;\omega}_{f,\lambda}\in \Pi_{x;\omega}$  and}
\[
\pi^{\star}_\lambda(x,f;\omega)=\bbE^{\bbQ^{x;\omega}_{f,\lambda}}\int_0^\infty e^{-\lambda t}f(X_t)\,dt,\quad
\pi_{\star}^\lambda(x,f;\omega)=\bbE^{\bbQ^{x;\omega}_{f,\lambda}}\int_0^\infty e^{-\lambda t}f(X_t)\,dt.
\]
From the above formulas, one  infers that
$\pi^{\star}_\lambda(x,f;\cdot)$ (resp. $\pi_{\star}^\lambda
(x,f;\cdot)$) is upper (resp. lower) semi-continuous.
Hence, both of these functions  {are Borel measurable in $\omega\in \Omega$}. Let $S$ be a countable and dense subset of $(0,\infty)\times \bbR^d\times C_0(\bbR^d)$. Observe that
\[\Omega^{\cal M}=\{\omega\in\Omega: |\pi^{\star}_\lambda(x,f;\omega)-
  \pi_{\star}^\lambda (x,f;\omega)|=0,\quad  (\lambda,x,f)\in S\}.\]
This implies  that {$\Omega^{\cal M}$ is a Borel measurable subset of $\Omega$.}

\end{document}